\documentclass[11pt,reqno]{amsart}

\usepackage{enumerate}
\usepackage{amsmath, amssymb, amsthm}
\usepackage{mathrsfs}
\usepackage{esint}
\usepackage{xcolor}
\usepackage{mathtools}
\usepackage{hyperref}
\usepackage{bm}

\numberwithin{equation}{section}

\newtheorem{theorem}{Theorem}[section]
\newtheorem{lemma}[theorem]{Lemma}
\newtheorem{corollary}[theorem]{Corollary}
\newtheorem{remark}[theorem]{\bf{Remark}}
\newtheorem{assumption}[theorem]{Assumption}

\newtheorem{definition}[theorem]{Definition}

\theoremstyle{remark}
\theoremstyle{definition}

\newcommand\bL{\mathbb{L}}

\newcommand\bR{\mathbb{R}}

\newcommand\bH{\mathbb{H}}

\newcommand\bZ{\mathbb{Z}}

\newcommand\bE{\mathbb{E}}
\newcommand\bN{\mathbb{N}}

\newcommand\cB{\mathcal{B}}
\newcommand\cC{\mathcal{C}}
\newcommand\cD{\mathcal{D}}
\newcommand\cF{\mathcal{F}}

\newcommand\cH{\mathcal{H}}

\newcommand\cL{\mathcal{L}}
\newcommand\cP{\mathcal{P}}

\newcommand\cS{\mathcal{S}}
\newcommand\cM{\mathcal{M}}

\newcommand\cbrk{\text{$]$\kern-.15em$]$}}
\newcommand\opar{\text{\,\raise.2ex\hbox{${\scriptstyle
|}$}\kern-.34em$($}}
\newcommand\cpar{\text{$)$\kern-.34em\raise.2ex\hbox{${\scriptstyle |}$}}\,}

\newcommand\ep{\varepsilon}

\begin{document}

\title[STFBEs driven by multiplicative space-time white noise]{$L_p$-solvability and H\"older regularity for stochastic time fractional Burgers' equations driven by multiplicative space-time white noise}

\author{Beom-Seok Han 
}

\address{Department of Mathematics, Pohang University of Science and Technology, 77, Cheongam-ro, Nam-gu, Pohang, Gyeongbuk, 37673, Republic of Korea}

\email{hanbeom@postech.ac.kr}

\thanks{This work was supported by the National Research Foundation of Korea (NRF) grant
funded by the Korea government (MSIT) (No. NRF-2021R1C1C2007792) 
and the BK21 Fostering Outstanding Universities for Research (FOUR) funded by the Ministry of Education (MOE, Korea) and the National Research Foundation of Korea (NRF)}



\begin{abstract}
We present the $L_p$-solvability for stochastic time fractional Burgers' equations driven by multiplicative space-time white noise:
$$ 
\partial_t^\alpha u = a^{ij}u_{x^ix^j} + b^{i}u_{x^i} + cu  + \bar b^i u u_{x^i} + \partial_t^\beta\int_0^t \sigma(u)dW_t,\,t>0;\,\,u(0,\cdot) = u_0,
$$
where $\alpha\in(0,1)$, $\beta < 3\alpha/4+1/2$, and  $d< 4 - 2(2\beta-1)_+/\alpha$.  The operators $\partial_t^\alpha$ and $\partial_t^\beta$ are the Caputo fractional derivatives of order $\alpha$ and $\beta$, respectively. The process $W_t$ is an $L_2(\bR^d)$-valued cylindrical Wiener process, and the coefficients $a^{ij}, b^i, c$ and $\sigma(u)$ are random.

In addition to the existence and uniqueness of a solution, we also suggest the H\"older regularity of the solution. For example, for any constant $T<\infty$, small $\ep>0$, and almost sure $\omega\in\Omega$, we have
$$ \sup_{x\in\bR^d}|u(\omega,\cdot,x)|_{C^{\left[ \frac{\alpha}{2}\left( \left( 2-(2\beta-1)_+/\alpha-d/2 \right)\wedge1 \right)+\frac{(2\beta-1)_{-}}{2} \right]\wedge 1-\ep}([0,T])}<\infty
$$
and
$$ \sup_{t\leq T}|u(\omega,t,\cdot)|_{C^{\left( 2-(2\beta-1)_+/\alpha-d/2 \right)\wedge1 - \ep}(\bR^d)} < \infty.
$$
The H\"older regularity of the solution in time changes behavior at $\beta = 1/2$. Furthermore, if $\beta\geq1/2$, then the H\"older regularity of the solution in time is $\alpha/2$ times the one in space.

\end{abstract}

\maketitle

\section{Introduction}
This article investigates the existence, uniqueness, $L_p$-regularity, and maximal H\"older regularity of a solution to stochastic time fractional Burgers' equations (STFBEs) driven by space-time white noise. We consider 
\begin{equation}
\label{main_equation}
\begin{aligned}
\partial_t^\alpha u = Lu  + \bar b^i u u_{x^i} + \partial_t^\beta\int_0^t \sigma(u)dW_t,\quad (\omega,t,x)\in\Omega\times(0,\infty)\times\bR^d;\quad u(0,\cdot) = u_0,
\end{aligned}
\end{equation}
where $\alpha\in(0,1)$, $\beta < \frac{3}{4}\alpha + \frac{1}{2}$, and $d < 4 - \frac{2 (2\beta-1)_+}{\alpha}$. The operators $\partial_t^\alpha$ and $\partial_t^\beta$ are the Caputo fractional derivatives of order $\alpha$ and $\beta$, and the operator $L$ is the second order  {\it random} differential operator defined as follows: 
\begin{equation*}
\begin{gathered}
(L u)(\omega,t,x)  = a^{ij}(\omega,t,x) u_{x^ix^j} + b^i(\omega,t,x) u_{x^i} + c(\omega,t,x) u.
\end{gathered}
\end{equation*}
The random coefficients $a^{ij}$, $b^i$, and $c$ are predictable, differentiable (or continuous), and bounded functions. The diffusion coefficient $\sigma(u) = \sigma(\omega,t,x,u)$ is a predictable and measurable function satisfying growth conditions and Lipschitz continuity in $u$. The detailed conditions on $a^{ij}$, $b^i$, $c$, and $\sigma$ are described in Assumptions \ref{assumptions_on_coefficients_deterministic_part} and \ref{assumptions_on_coefficients_stochastic_part}. The random measure $dW_t$ is induced from an $L_2(\bR^d)$-valued cylindrical Wiener process $W_t$.

When $\alpha = \beta = 1$ in equation \eqref{main_equation}, the equation is said to be a stochastic Burgers' equation (SBE) of form 
\begin{equation}
\label{SBE}
\partial_t u = Lu  + \bar b u u_{x} +  \sigma(u)\dot W,\quad (\omega,t,x)\in\Omega\times(0,\infty)\times\bR;\quad u(0,\cdot) = u_0,
\end{equation}
where $\dot W$ is the space-time white noise.
Numerous studies have been conducted on the equation \eqref{SBE}, but we only refer to the reader to \cite{gyongy1998existence,gyongy1999stochastic,lewis2018stochastic}. In \cite{gyongy1998existence}, the author proved the uniqueness, existence, and continuity of a solution to a semilinear equation, including an equation of type \eqref{SBE} on the unit interval $(0,1)$. Additionally, the same properties of a solution on $\bR$ were obtained in \cite{gyongy1999stochastic} when the $L_2$ bounded conditions on $\sigma(u)$ were imposed. In \cite{lewis2018stochastic}, the authors investigated the H\"older regularity and moment estimates of the random field solution to \eqref{SBE} with $L = \Delta$ and $\bar b = -1$.

In contrast, (deterministic) partial differential equations with Caputo fractional derivatives have been used in many fields, such as electrochemical processes \cite{bagley1985fractional,ichise1971analog}, dielectric polarization \cite{sun1984application}, viscoelastic materials \cite{podlubny1998fractional}, biology \cite{magin2004fractional}, and physics \cite{giona1992fractional,hilfer2000applications}. Especially, equation \eqref{main_equation} with $\alpha\in(0,1)$ and $\sigma(u) = 0$ is called a time fractional Burgers' equation (TFBE), which describes the propagation of waves through viscous media (\cite{garra2011fractional,keller1981propagation}). Indeed, various researches have been conducted on numerical analysis for the TFBE (see \cite{akram2020efficient,el2012parametric,esen2015numerical,inc2008approximate,li2016linear}). From a mathematical standpoint, it is reasonable to wonder whether it is possible to demonstrate the uniqueness and existence of a solution to STFBE \eqref{main_equation}, and also to obtain the H\"older regularity of the solution. To the best of our knowledge, \cite{zou2017stochastic} is the only study that answers this question. The authors of \cite{zou2017stochastic} demonstrate the existence, uniqueness, and regularity of the {\it mild} solution to SBEs with fractional derivatives in time and space on a bounded domain $\cD\subset \bR^d$.

In this paper, we provide the $L_p$ uniqueness, existence, and regularity of a \textit{strong} solution to equation \eqref{main_equation} with {\it random} second order differential operator $L$ on the {\it whole spatial domain} $\bR^d$. Additionally, we achieve the H\"older regularity of the solution in time and space. In detail, if $u(\omega,t,x)$ denotes the solution to equation \eqref{main_equation}, then for any bounded stopping time $\tau\leq T$ and small constant $\ep>0$, almost surely,
\begin{equation}
\begin{gathered}
\sup_{x\in\bR^d}|u(\omega,\cdot,x)|_{C^{\left[ \frac{\alpha}{2}\left( \left( 2-(2\beta-1)_+/\alpha-d/2 \right)\wedge1 \right)+\frac{(2\beta-1)_{-}}{2} \right]\wedge 1-\ep}([0,\tau])} < \infty, \label{intro: holder reg of sol} \\
\sup_{t\leq\tau}|u(\omega,t,\cdot)|_{C^{\left( 2-(2\beta-1)_+/\alpha-d/2 \right)\wedge1 - \ep}(\bR^d)} < \infty.
\end{gathered}
\end{equation}
where $a_{+} = (|a|+a)/2$, $a_{-} = (|a|-a)/2$, and $C^\gamma(\cD)$ is the H\"older spaces. Observe that the behavior of the H\"older regularity of the solution in time changes from $\beta = 1/2$. For example, if $\beta \geq 1/2$, then the H\"older regularity of the solution in time is $\alpha/2$ times that of the regularity in space. Additionally, we can recover the the H\"older regularity results of SBEs by letting $\alpha,\beta \uparrow 1$. These results are consistent with the well-known results of stochastic heat equations driven by space-time white noise (e.g. \cite[Remark 8.7]{kry1999analytic} or \cite[Corollary 3.1]{han2022regularity}). In contrast, if $\beta < 1/2$, the H\"older regularity in time gains additional regularity by as much as $1/2-\beta$. (Remark \ref{regularity comparision}).

Several remarks about the proof are made. The proof strategy for the main theorem (Theorem \ref{main theorem}) is based on \cite{han2022regularity}. However, some differences exist because since it is not certain that It\^o's formula and the maximum principle hold for STFBE \eqref{main_equation}.

Thus, the proof proceeds as follows. As in \cite{han2022regularity}, we focus on proving the uniqueness and existence of the $L_p$ solution in each (large) $p>2$, and the main difficulty is to demonstrating the existence of the solutions. Hence, we consider the cut-off form of equation \eqref{main_equation} to obtain local solutions. Afterward, we construct a global solution candidate $u$ by pasting the local solutions (Lemma \ref{lemma:existence uniqueness reg of local solution} and Remark \ref{def_u}). A uniform $L_p$ bound of $u$ is required to show that our the candidate $u$ is a global solution; thus, we divide the local solution into two parts: the noise-dominating and the nonlinear-dominating part. To estimate the noise-dominating parts, we employ the $L_p$ bound of the diffusion coefficient $\sigma(u)$ (Lemma \ref{lemma:noise dominant part}). In contrast, to control the nonlinear-dominating part (Lemma \ref{lemma:nonlinear control}), we employ an inequality similar to the chain rule (Lemma \ref{lemma: chain ineq}) and a version of the Gr\"onwall inequality including the Caputo fractional derivatives (Theorem \ref{time fractional integral gronwall ineq}).

To obtain the maximal H\"older regularity of the solution to equation \eqref{main_equation}, we require two components: the H\"older embedding theorem for the solution space $\cH_p^\gamma(\tau)$ (Theorem \ref{thm:embedding theorem for sol space}) and the uniqueness of the solution in $p$ (Theorem \ref{uniqueness_in_p}). Indeed, when the $L_p$ existence and uniqueness of a solution are given, we have the H\"older regularity of the solution in each (large) $p>2$ by employing the H\"older embedding theorem for the solution space (Theorem \ref{thm:embedding theorem for sol space} and Theorem \ref{main theorem}). The H\"older regularity of the solution becomes larger as a large $p$ is chosen; thus, we have to select $p$ that is as large as possible. Therefore, we require the uniqueness of solutions in $p$ because $p$ varies.

This article is organized as follows. Section \ref{Preliminaries} introduces the definitions and properties of space-time white noise, fractional calculus, and stochastic Banach spaces. Additionally, we present the H\"older embedding theorem for the solution space $\cH_p^\gamma(\tau)$. Section \ref{Main results} states the main results of this article and suggests some remarks. The proof of the main results is presented in Section \ref{section:proof of main thm}. Next, Section \ref{section: proof of continuity of the solution in t} proves the H\"older embedding theorem for the solution space $\cH_p^\gamma(\tau)$.

We finish this section with an introduction to the notation used in this paper. The sets $\bN$ and $\bR$ are sets of natural and real numbers, respectively. The set $\bR^d$ denotes the $d$-dimensional Euclidean space of points $x = (x^1,\dots,x^d)$ for $x^i\in\bR$. Throughout this paper, we assume Einstein’s summation convention on $i,j,k\in\bN$. We use $:=$ to denote a definition. For a real-valued function $f$, we set the following:
\begin{equation*}
	f_+:=\frac{ | f | + f }{2} \quad \text{and} \quad	f_-:=\frac{ | f | - f }{2}.
\end{equation*}
For a normed space $F$, a measure space $(X,\cM,\mu)$, and $p\in[1,\infty)$, a space $L_p(X,\cM,\mu;F)$ is a set of $F$-valued $\cM^\mu$-measurable functions such that
\begin{equation*}
	\| u \|_{L_p(X,\cM,\mu;F)}:=\left( \int_{X}\| u(x) \|_F^p \mu(dx) \right)^{1/p} < \infty.
\end{equation*}
A set $\mathcal{M}^{\mu}$ is the completion of $\cM$ with respect to the measure $\mu$. For $\gamma\in(0,1]$ and $k = 0,1,2,\dots$, a set $C^{k+\gamma}(\bR^d)$ is the set of $\bR$-valued continuous functions $u = u(x)$ such that
$$ |u|_{C^{\gamma+k}(\bR^d)}:=\sup_{x\in \bR^d,|\bm{\beta}| = k}\left| D^{\bm{\beta}} u(x) \right| + \sup_{\substack{x,y\in\bR^d, x\neq y \\ |\bm{\beta}| = k}}\frac{\left| D^{\bm{\beta}} u(x) - D^{\bm{\beta}} u(y) \right|}{|x-y|^{\gamma}}<\infty,
$$
where $\bm{\beta}$ is a multi-index.
Similarly, for $\gamma\in(0,1]$ and $0 < T < \infty$, the set $C^{\gamma}([0,T];F)$ is the set of $F$-valued continuous functions $u$ such that
$$ |u|_{C^{\gamma}([0,T];F)}:=\sup_{t\in [0,T]}\left|  u(t) \right|_{F}+\sup_{\substack{t,s\in[0,T], \\ s\neq t}}\frac{\left|  u(t) -  u(s) \right|_F}{|t-s|^{\gamma}}<\infty.
$$

For $a,b\in \bR$, we set
$a \wedge b := \min\{a,b\}$ and $a \vee b := \max\{a,b\}$.
Let $\cS = \cS(\bR^d)$ denote the set of Schwartz functions on $\bR^d$. Let $N = N(a_1,a_2,...,a_k)$ be a generic constant if $N$ depends only on $a_1,a_2,...,a_k$. The constant $N$ can vary line by line. For functions depending on $\omega$, $t$, and $x$, the argument $\omega \in \Omega$ is omitted. Finally, for $x\in\bR^d$, $\bar x^i := (x^1,\dots,x^{i-1},x^{i+1},\dots,x^d)$.

\vspace{2mm}

\section{Preliminaries}
\label{Preliminaries}

In this section, we introduce the definitions and properties of space-time white noise, fractional calculus, and stochastic Banach spaces. Throughout this paper, $(\Omega, \cF, P)$ is a complete probability space equipped with a filtration $\{\cF_t\}_{t\geq0}$. Let $\{\cF_t\}_{t\geq0}$ denote a filtration satisfying the usual conditions. Let $\cP$ be the predictable $\sigma$-field related to $\{\cF_t\}_{t\geq0}$.

First, we present the space-time white noise $\dot W$ to understand the stochastic part of \eqref{main_equation}.
\begin{definition}[Space-time white noise]
\label{def if STWN}
A generalized random field $\dot W$ is said to be the space-time white noise if it is a centered Gaussian random field such that its covariance is given by
\begin{equation*} 
\bE\, \dot W(h)\, \dot W(g) = \int_0^\infty \int_{\bR^d} h(t,x) g(t,x) dxdt,\quad \forall h,g \in L_2((0,\infty)\times\bR^d).
\end{equation*}
\end{definition}

\begin{remark} \label{remark:representation}
\label{rmk: convert Walsh integral to series of ito integral} 

We employ a series of It\^o's stochastic integral to interpret the stochastic part of equation \eqref{main_equation}. More precisely, let $\{ \eta^k:k\in\bN \}$ be an orthonormal basis on $L_2(\bR^d)$. If we define
\begin{equation*}
	w_t^k:= \int_0^t \int_{\bR^d} \eta^k(x)\dot W(ds,dx)
\end{equation*}
using the Walsh integral (see \cite{walsh1986introduction}), then $\{w_t^k:k\in\bN\}$ is a set of one dimensional independent Wiener processes. Then, if we set (see \cite[Section 8.3]{kry1999analytic}, and \cite[Section 7]{kim2019sobolev})
\begin{equation*}
W_t := \sum_{k=1}^\infty \eta^k w_t^k,
\end{equation*}
then $W_t$ is an $L_2(\bR^d)$-valued cylindrical Wiener process and $dW_t = \sum_k \eta^k dw_t^k$. Thus, equation \eqref{main_equation} can be rewritten as
\begin{equation*}
\partial_t^\alpha u = Lu  + \bar b^i u u_{x^i} + \partial_t^\beta\int_0^t \sigma(u) \eta^k dw^k_t,\quad (\omega,t,x)\in\Omega\times(0,\infty)\times\bR^d;\quad u(0,\cdot) = u_0.
\end{equation*}

\end{remark}

\vspace{2mm}

Next, we review the facts of fractional calculus. For more information, we refer to the reader to \cite{baleanu2012fractional,herrmann2011fractional,kilbas1993fractional,podlubny1998fractional}.

\begin{definition}
Let $\alpha>0$, and for $\varphi\in L_1((0,T))$, the Riemann-Liouville fractional integral of the order $\alpha$ is defined as follows: 
$$ I_t^\alpha\varphi(t) := (I^{\alpha}_{t}\varphi)(t) := \frac{1}{\Gamma(\alpha)}\int_0^t(t-s)^{\alpha-1}\varphi(s)ds\quad \mbox{for all} \quad t\in (0,T),
$$
where $\Gamma(\alpha):=\int_0^\infty t^{\alpha - 1}e^{-t}dt$. 
\end{definition}

\begin{remark}
For any $q \in [1,\infty]$, by Jensen's inequality
\begin{equation}
\label{ineq:lp ineq for time fractional int}
\|I^\alpha \varphi \|_{L_q((0,T))}\leq N(\alpha,p,T)\| \varphi \|_{L_q((0,T))}.
\end{equation}
Therefore, $I_t^\alpha \varphi(t)$ is well-defined and finite for almost all $t\leq T$.
Additionally, Fubini's theorem implies that, for $\alpha,\beta\geq0$, we have
\begin{equation*}
I^{\alpha+\beta}\varphi(t)=I^{\alpha}I^{\beta}\varphi(t).
\end{equation*}

\end{remark}

\begin{definition}
\label{def of time fractional derivative}
For $\alpha>0$, let $n\in\bN$ be a nonnegative integer such that $n-1 \leq \alpha < n$. Suppose $\varphi(t)$ is a real-valued function on $[0,T]$ such that $\varphi$ is $(n-1)$-times differentiable and $(\frac{d}{dt})^{n-1}\varphi$ is absolutely continuous on $[0,T]$. 
\begin{enumerate}[(i)]
\item 
The Riemann-Liouville fractional derivative $D_t^\alpha \varphi$ is defined as 
\begin{equation*}
D_t^\alpha \varphi(t) := \frac{1}{\Gamma(n-\alpha)}\frac{d^n}{dt^n}\int_0^t (t-s)^{n-\alpha-1}\varphi(s)ds.
\end{equation*}

\item
The Caputo fractional derivative $\partial_t^\alpha\varphi$ is defined as 
\begin{equation*}
\begin{aligned}
\partial_t^\alpha \varphi 
&:= \frac{1}{\Gamma(n-\alpha)}\int_0^t (t-s)^{n-\alpha-1}\varphi^{(n)}(s)ds \\
&:= \frac{1}{\Gamma(n-\alpha)}\frac{d}{dt}\int_0^t (t-s)^{n-\alpha-1}\left[ \varphi^{(n-1)}(s)-\varphi^{(n-1)}(0) \right]ds. \\
\end{aligned}
\end{equation*}

\end{enumerate}
\end{definition}

\begin{remark}
\label{rmk:prop of fractional calculus}
\begin{enumerate}[(i)]

\item
\label{exchange of fractional integral and derivative}
For any $\alpha,\beta \geq0$, $D^\alpha_t D^\beta_t\varphi = D^{\alpha+\beta}_t\varphi$  and
\begin{equation*}
D^\alpha_t I^\beta_t \varphi = D^{\alpha-\beta}_t  \varphi 1_{\alpha > \beta} + I_t^{\beta-\alpha}\varphi 1_{\alpha\leq\beta}.
\end{equation*}
Additionally, if $\alpha\in(0,1)$, $I^{1-\alpha}_t \varphi$ is absolutely continuous, and $I^{1-\alpha}_t\varphi(0) = 0$, then the following equality holds:
\begin{equation*}
I^\alpha_t D^\alpha_t\varphi(t) = \varphi(t).
\end{equation*}

\item 
By the definition of fractional derivatives, if $\varphi(0) = \varphi^{(1)}(0) = \cdots = \varphi^{(n -1)}(0) = 0$, then $D^\alpha_t\varphi = \partial_t^\alpha \varphi$.

\end{enumerate}

\end{remark}

\vspace{2mm}

Below we recall the definitions and properties of stochastic Banach spaces (for more detail, see \cite{grafakos2009modern,kim2020sobolev,kry1999analytic,krylov2008lectures}). The solution space $\cH_{p}^{\gamma}(T)$ and embedding theorems for $\cH_{p}^{\gamma}(T)$ are suggested.

\begin{definition}
Let $p>1$ and $\gamma \in \mathbb{R}$. The space $H_p^\gamma=H_p^\gamma(\bR^d)$ is the set of all tempered distributions $u$ on $\mathbb{R}$ such that
$$ \| u \|_{H_p^\gamma} := \left\| (1-\Delta)^{\gamma/2} u \right\|_{L_p} = \left\| \cF^{-1}\left[ (1+|\xi|^2)^{\gamma/2}\cF(u)(\xi)\right]\right\|_{L_p}<\infty.
$$
Similarly, $H_p^\gamma(l_2) = H_p^\gamma(\bR^d;l_2)$ is a space of $l_2$-valued functions $g=(g^1,g^2,\cdots)$ such that 
$$ \|g\|_{H_{p}^\gamma(l_2)}:= \left\| \left| (1-\Delta)^{\gamma/2} g \right|_{l_2}\right\|_{L_p} = \left\| \left|\cF^{-1}\left[ \left( 1+|\xi|^2 \right)^{\gamma/2}\cF(g)(\xi)\right] \right|_{l_2} \right\|_{L_p}
< \infty. 
$$
\end{definition}

\begin{remark}
\label{Kernel}
Let $d\in\bN$ and $\gamma \in (0,\infty)$. A nonnegative smooth function $R_{\gamma}(x)$ exists on $\bR^d$ such that, for $u\in C_c^\infty(\bR^d)$,
\begin{equation*}
\left( (1-\Delta)^{-\gamma/2}\, u \right)(x)= \int_{\bR^d} R_\gamma(y)u(x-y)dy
\end{equation*}
 and 
\begin{equation*}
\left| R_\gamma(x) \right| \leq N A_{\gamma,d}(x)1_{|x|\leq2} + Ne^{-|x|/2}1_{|x|\geq2},
\end{equation*}
where $N = N(\gamma,d)$ is a positive constant and
\begin{equation*}
A_{\gamma,d}(x) = 
\begin{cases}
|x|^{\gamma-d}+1+O(|x|^{\gamma-d+2})\quad&\mbox{for}\quad0<\gamma<d, \\
\log(2/|x|)+1+O(|x|^{2})\quad&\mbox{for}\quad \gamma=d, \\
1+O(|x|^{\gamma-d})\quad&\mbox{for}\quad \gamma > d.
\end{cases}
\end{equation*}
For more detail, see \cite[Proposition 1.2.5]{grafakos2009modern}.
\end{remark}

We introduce the space of point-wise multipliers in $H_p^\gamma$.

\begin{definition}
\label{def_pointwise_multiplier}
Fix $\gamma\in\bR$ and $\alpha\in[0,1)$ such that $\alpha = 0$ if $\gamma\in\bZ$ and $\alpha>0$ if $|\gamma|+\alpha$ is not an integer. Define
\begin{equation*}
\begin{aligned}
B^{|\gamma|+\alpha} = 
\begin{cases}
B(\bR) &\quad\text{if } \gamma = 0, \\
C^{|\gamma|-1,1}(\bR) &\quad\text{if $\gamma$ is a nonzero integer}, \\
C^{|\gamma|+\alpha}(\bR) &\quad\text{otherwise},
\end{cases}
\end{aligned}
\end{equation*}
\begin{equation*}
\begin{aligned}
B^{|\gamma|+\alpha}(\ell_2) = 
\begin{cases}
B(\bR,\ell_2) &\quad\text{if } \gamma = 0, \\
C^{|\gamma|-1,1}(\bR,\ell_2) &\quad\text{if $\gamma$ is a nonzero integer}, \\
C^{|\gamma|+\alpha}(\bR,\ell_2) &\quad\text{otherwise},
\end{cases}
\end{aligned}
\end{equation*}
where $B(\bR)$ is the space of bounded Borel functions on $\bR$, $C^{|\gamma|-1,1}(\bR)$ represents the space of $|\gamma|-1$ times continuous differentiable functions whose derivatives of the $(|\gamma|-1)$th order derivative are Lipschitz continuous, and $C^{|\gamma|+\alpha}$ is the real-valued H\"older spaces. The space $B(\ell_2)$ denotes a function space with $\ell_2$-valued functions instead of real-valued function spaces.

\end{definition}

Below we collect the properties of Bessel potential spaces.

\begin{lemma}
\label{lemma:properties of bessel potential spaces}

Let $\gamma \in \bR$ and $p>1$. 
\begin{enumerate}[(i)]
\item 
\label{dense_subset_bessel_potential} 
The space  $C_c^\infty(\bR^d)$ is dense in $H_{p}^{\gamma}$. 

\item
\label{item:Holder embedding for bessel potential space}
Let $\gamma - d/p = n+\nu$ for some $n=0,1,\cdots$ and $\nu\in(0,1]$. Then, for any  $k\in\{ 0,1,\cdots,n \}$, we have
\begin{equation} 
| D^k u |_{C(\bR^d)} + | D^n u |_{\cC^\nu(\bR^d)} \leq N \| u \|_{H_{p}^\gamma},
\end{equation}
where $\cC^\nu(\bR^d)$ is the Zygmund space.

\item
\label{bounded_operator}
The operator $D_i:H_p^{\gamma}\to H_p^{\gamma+1}$ is bounded. Moreover, for any $u\in H_p^{\gamma+1}$,
$$ \left\| D^i u \right\|_{H_p^\gamma} \leq N\| u \|_{H_p^{\gamma+1}},
$$
where $N = N(\gamma,p)$.

\item 
\label{bounded operator2}
For $\gamma_1,\gamma_2\in\bR$, and $u\in H_{p}^{\gamma_1+\gamma_2}$, we have
\begin{equation*}
\| \Delta^{\gamma_1/2}u \|_{H_p^{\gamma_2}} \leq N\| u \|_{H_p^{\gamma_1+\gamma_2}},
\end{equation*}
where $N = N(\gamma_1,\gamma_2)$.

\item 
\label{multiplier theorem}
For $\gamma\in(0,1)$, and $u\in H_{p}^{\gamma}$, we have
\begin{equation*}
\| (1- \Delta^{\gamma})u \|_{L_p} \leq N\left( \| u \|_{L_p} + \| (-\Delta)^{\gamma} u \|_{L_p} \right),
\end{equation*}
where $N = N(\gamma,p)$.

\item 
\label{item:isometry}
For any $\mu,\gamma\in\bR$, the operator $(1-\Delta)^{\mu/2}:H_p^\gamma\to H_p^{\gamma-\mu}$ is an isometry.

\item
\label{multi_ineq}
Let 
\begin{equation*} 
\begin{gathered}
\ep\in[0,1],\quad p_i\in(1,\infty),\quad\gamma_i\in \bR,\quad i=0,1,\\
\gamma=\ep\gamma_1+(1-\ep)\gamma_0,\quad1/p=\ep/p_1+(1-\ep)/p_0.
\end{gathered}
\end{equation*}
Then, we have
\begin{equation*}
\|u\|_{H^\gamma_{p}} \leq \|u\|^{\ep}_{H^{\gamma_1}_{p_1}}\|u\|^{1-\ep}_{H^{\gamma_0}_{p_0}}.
\end{equation*}

\item \label{pointwise_multiplier}
Let $u\in H_p^\gamma$. Then, we have
\begin{equation*}
\| au \|_{H_p^\gamma} \leq N\| a \|_{B^{|\gamma|+\alpha}}\| u \|_{H_p^\gamma}\quad\text{and}\quad\| bu \|_{H_p^\gamma(\ell_2)} \leq N\| b \|_{B^{|\gamma|+\alpha}(\ell_2)}\| u \|_{H_p^\gamma},
\end{equation*}
where $N = N(\gamma,p)$ and $B^{|\gamma|+\alpha},B^{|\gamma|+\alpha}(\ell_2)$ are introduced in Definition \ref{def_pointwise_multiplier}.

\end{enumerate}

\end{lemma}
\begin{proof}
The above results are well known. For \eqref{dense_subset_bessel_potential}, \eqref{bounded_operator}, \eqref{item:isometry}, and \eqref{multi_ineq}, see Theorems 13.3.7 (i), 13.8.1, 13.3.7 (ii), and Exercise 13.3.20 of \cite{krylov2008lectures}, respectively.
In the case of \eqref{item:Holder embedding for bessel potential space} and \eqref{bounded operator2}, see \cite{triebel2010theory}. For \eqref{multiplier theorem}, see Theorems 1.3.6 and 1.3.8 of \cite{grafakos2009modern}. For \eqref{pointwise_multiplier}, we refer the reader to \cite[Lemma 5.2]{kry1999analytic}.

\end{proof}

\begin{definition}[Stochastic Banach spaces]
Let $\tau\leq T$ be a bounded stopping time, $p \geq 2$, and $\gamma\in\bR$. Set $\opar0,\tau\cbrk:=\{ (\omega,t):0<t\leq \tau(\omega) \}$ and define
\begin{gather*}
\bH_{p}^{\gamma}(\tau) := L_p\left(\opar0,\tau\cbrk, \mathcal{P}, dP \times dt ; H_{p}^\gamma\right), \\
\bH_{p}^{\gamma}(\tau,l_2) := L_p\left(\opar0,\tau\cbrk,\mathcal{P}, dP \times dt;H_{p}^\gamma(l_2)\right),\\
U_p^{\gamma} := L_p\left( \Omega,\cF_0,H_p^{\gamma} \right).
\end{gather*}
We write $u\in \bH_p^\gamma$ if $u\in \bH_p^{\gamma}(\tau)$ exists for any bounded stopping time $\tau$. Additionally, if $\gamma = 0$, then we use $\bL$ instead of $\bH$, $\|f\|_{\bL_p(\tau)} := \|f\|_{\bH_p^0(\tau)}$.
The norm of each space is defined naturally, for example,
\begin{equation*}
\|f\|_{\bH_p^\gamma(\tau)} := \left( \bE\int_0^\tau \| f(t,\cdot) \|_{H_p^\gamma}^p \, dt\right)^{1/p}.
\end{equation*}

\end{definition}

\vspace{2mm}

Lemma \ref{prop of time fractional integral and stochastic integral} exhibits the relation between the stochastic and fractional integrals, which is employed when $I_t^\alpha$ or $D_t^\alpha$ is applied to the stochastic part of the SPDEs. 

\begin{lemma}
\label{prop of time fractional integral and stochastic integral}
Let $T<\infty$ be a constant.
\begin{enumerate}[(i)]
\item 
\label{summation time fractional integral change}
Let $\alpha \geq 0$ and $h\in L_2(\Omega\times[0,T],\cP;l_2)$. Then, the equality
\begin{equation*}
I^\alpha\left( \sum_{k=1}^\infty \int_0^\cdot h^k(s) dw_s^k \right)(t) = \sum_{k=1}^\infty \left( I^\alpha\int_0^\cdot h^k(s)dw_s^k \right)(t)
\end{equation*}
holds for all $t\leq T$ almost surely and in $L_2(\Omega\times[0,T])$, where the series on both sides converge in probability.

\item 
\label{convergence of sequence with time fractional integral and stochastic integral}
If $\alpha\geq0$ and $h_n\to h$ in $L_2(\Omega\times[0,T],\cP;l_2)$ as $n\to\infty$, then
\begin{equation*}
\sum_{k = 1}^\infty \left( I^\alpha\int_0^\cdot h_n^k dw_s^k \right)(t) \to \sum_{k=1}^\infty \left( I^\alpha \int_0^\cdot h^k dw_s^k \right)(t)
\end{equation*}
in probability uniformly on $[0,T]$.

\item 
\label{time fractional derivative and stochastic integral}
If $\alpha > 1/2$ and $h\in L_2(\Omega\times[0,T],\cP;l_2)$, then $\left( I^\alpha \sum_{k=1}^\infty\int_0^\cdot h^k (s) dw_s^k \right)(t)$ is differentiable in $t$ and
\begin{equation*}
\frac{\partial}{\partial t}\left( I^\alpha \sum_{k=1}^\infty\int_0^\cdot h^k (s) dw_s^k \right)(t) = \frac{1}{\Gamma(\alpha)}\sum_{k=1}^\infty \int_0^t (t-s)^{\alpha-1}h^k(s)dw_s^k
\end{equation*}
(a.e.) on $\Omega\times[0,T]$.

\end{enumerate}

\end{lemma}
\begin{proof}
See Lemmas 3.1 and 3.3 of \cite{chen2015fractional}.
\end{proof}

\vspace{2mm}

Fix a small $\kappa_0>0$. For $\alpha\in(0,1)$ and $\beta < \alpha+1/2$, set
\begin{equation}
\label{def of c0}
c_0:= \frac{(2\beta-1)_+}{\alpha} + \kappa_0 1_{\beta = 1/2}.
\end{equation}

Next, we introduce the solution spaces (for more detail, see Definitions 2.9 and 2.12 in \cite{kim2020sobolev}).

\begin{definition}[Solution spaces] 
\label{def:solution space}
Let $\tau\leq T$ be a bounded stopping time, $\alpha\in(0,1)$, $\beta < \alpha+1/2$, $\gamma\in \bR$, and $ p\geq 2$.

\begin{enumerate}[(i)]
\item 
For $u\in \bH_p^{\gamma}(\tau)$, we write $u\in\cH^{\gamma}_p(\tau)$ if $u_0\in U_p^{\gamma}$, $f\in
\bH_{p}^{\gamma-2}(\tau)$, and $g\in
\bH_{p}^{\gamma-2+c_0}(\tau,l_2)$ such that 
\begin{equation*}
\begin{aligned}
\partial_t^\alpha u(t,x) = f(t,x) + \partial_t^\beta\int_0^t g^k(s,x)dw_s^k,\quad 0<t\leq \tau; \quad u(0,\cdot) = u_0
\end{aligned}
\end{equation*}
in the sense of distribution. In other words, for any $\phi\in \cS$, the equality
\begin{equation}
\label{eq:def_of_sol}
\begin{aligned}
(u(t,\cdot),\phi) &= (u_0, \phi) + I_t^\alpha ( f,\phi ) + I_t^{\alpha-\beta} \sum_{k=1}^\infty  \int_0^t ( g^k(s,\cdot), \phi) dw_s^k \\
\end{aligned}
\end{equation}
holds for a.e. $(\omega,t)\in\Omega\times[0,\tau]$. 
If $\alpha-\beta \in(-1/2,0)$, we regard $I_t^{\alpha-\beta}$ as $\frac{\partial}{\partial t}I_t^{\alpha-\beta + 1}$. The norm in $\cH_{p}^{\gamma}(\tau)$ is defined as follows:
\begin{equation}
\| u \|_{\cH_{p}^{\gamma}(\tau)} :=  \| u \|_{\mathbb{H}_{p}^{\gamma}(\tau)} + \| u_0 \|_{U_{p}^{\alpha,\gamma}} + \inf_{f,g}\left\{ \| f \|_{\mathbb{H}_{p}^{\gamma-2}(\tau)} + \| g \|_{\mathbb{H}_{p}^{\gamma-2+c_0}(\tau,l_2)} \right\}.
\end{equation}

\item
We say $u \in \cH_{p,loc}^{\gamma}(\tau)$ if there exists a sequence $\tau_n\uparrow \tau$ such that $u\in \cH_{p}^{\gamma}(\tau_n)$ for each $n$. We write $u = v$ in $\cH_{p,loc}^{\gamma}(\tau)$ if a sequence of bounded stopping times $\tau_n\uparrow \tau$ exists such that $u = v$ in $\cH_{p}^{\gamma}(\tau_n)$ for each $n$. We omit $\tau$ if $\tau = \infty$. In other words,  $\cH_{p,loc}^{\gamma}=\cH_{p,loc}^{\gamma}(\infty).$
\end{enumerate}

\end{definition} 

\begin{remark}
If $\alpha-\beta \geq 0$, the stochastic part of \eqref{eq:def_of_sol} is considered
\begin{equation*}
\begin{aligned}
I_t^{\alpha-\beta} \sum_{k=1}^\infty  \int_0^t ( g^k(s,\cdot), \phi) dw_s^k
&=  \sum_{k=1}^\infty I_t^{\alpha-\beta} \int_0^t ( g^k(s,\cdot), \phi) dw_s^k.
\end{aligned}
\end{equation*}
Otherwise, if $\alpha-\beta \in(-1/2,0)$, we regard $I_t^{\alpha-\beta}$ as $\frac{\partial}{\partial t}I_t^{\alpha-\beta + 1}$. Then, by Lemma \ref{prop of time fractional integral and stochastic integral} \eqref{time fractional derivative and stochastic integral}, the stochastic part of \eqref{eq:def_of_sol} is
\begin{equation*}
\begin{aligned}
I_t^{\alpha-\beta}\left( \sum_{k=1}^\infty  \int_0^t ( g^k(s,\cdot), \phi) dw_s^k \right) 
&= \frac{\partial}{\partial t}\left( I^{\alpha-\beta+1} \sum_{k=1}^\infty  \int_0^t ( g^k(s,\cdot), \phi) dw_s^k \right) \\
&= \frac{1}{\Gamma(\alpha-\beta+1)}\sum_{k=1}^\infty\int_0^t(t-s)^{\alpha-\beta+1}(g^k(s,\cdot),\phi)dw_s^k.
\end{aligned}
\end{equation*}

\end{remark}

Below, we provide the properties of the solution space $\cH_p^\gamma(\tau)$.

\begin{theorem}
\label{prop of sol space}
Let $\tau\leq T$ be a bounded stopping time.
\begin{enumerate}[(i)]

\item
For $\nu\in\bR$, the map $(1-\Delta)^{\nu/2}:\cH_p^{\gamma+2}(\tau)\to\cH_p^{\gamma-\nu+2}(\tau)$ is an isometry.

\item 
If $\gamma\in\bR$, $\alpha\in(0,1)$, $\beta<\alpha+1/2$, and $p \geq 2$, then $\cH_p^{\gamma}(\tau)$ is a Banach space with the norm $\| \cdot \|_{\cH_p^{\gamma}(\tau)}$.

\end{enumerate}
\end{theorem}
\begin{proof}
Despite the fact that the definition of $U_p^\gamma$ is distinct, the proof is a repeat of \cite[Theorem 2.14]{kim2020sobolev} with $\tau$ instead of $T$.

\end{proof}

Next, we suggest the H\"older embedding theorems for $u\in\cH_p^\gamma(\tau)$. The proof of Theorem \ref{thm:embedding theorem for sol space} is contained in Section \ref{section: proof of continuity of the solution in t}.

\begin{theorem} 
\label{thm:embedding theorem for sol space}
Let $\tau\leq T$ be the bounded stopping time, $\gamma\in\mathbb{R}$, $\alpha\in(0,1)$, $\beta<\alpha+1/2$, and 
\begin{equation} 
\label{condition of p to be a solution}
    p >  2\vee \frac{1}{\alpha}\vee \frac{1}{\alpha-\beta+1/2}.
\end{equation}
Suppose $u\in \cH_p^{\gamma}(\tau)$. Assume $\alpha$, $\beta$, $\mu$, and $\nu$ satisfy
\begin{equation}
\label{conditions of mu and nu}
\frac{1}{\alpha p} < \mu < \frac{(\alpha(\nu+c_0/2)-\beta)\wedge1/2+1/2}{\alpha}\quad\text{and}\quad\frac{1}{\alpha p} < \nu <  1 - \frac{c_0}{2},
\end{equation}
where $c_0$ is the constant introduced in \eqref{def of c0}.
Then, for $u\in C^{\alpha\mu - 1/p}([0,\tau];H_p^{\gamma-2\nu})$ almost surely and
\begin{equation}
\label{item:time holder continuity of sol}
\bE\| u \|^p_{C^{\alpha\mu-1/p}\left([0,\tau];H_p^{\gamma-2\nu}\right)} \leq N\| u \|^p_{\cH_p^{\gamma}(\tau)},
\end{equation}
where $N = N(\alpha,\beta,\gamma,d,p,T)$.

\end{theorem}

\begin{remark}
Theorem \ref{thm:embedding theorem for sol space} is consistent with the previous results (\cite[Theorem 7.2]{kry1999analytic}). In other words, if we let $\alpha,\beta \uparrow 1$ in Theorem \ref{thm:embedding theorem for sol space}, conditions \eqref{condition of p to be a solution} and \eqref{conditions of mu and nu}, and the result \eqref{item:time holder continuity of sol} approach those of the case of $\alpha = \beta = 1$.
\end{remark}

By combining Lemma \ref{lemma:properties of bessel potential spaces} \eqref{item:Holder embedding for bessel potential space} and Theorem \ref{thm:embedding theorem for sol space}, we have the H\"older embedding results of solution space $\cH_p^{(2-c_0-d/2)\wedge 1}(\tau)$ which is a preparation to obtain the maximum H\"older regularity of solutions.

\begin{corollary}
\label{corollary:continuous embedding of solution}
Let $\tau\leq T$ be a bounded stopping time, $\alpha\in(0,1)$, $\beta < \alpha + 1/2$, and $0 < \gamma <  (2-c_0-d/2)\wedge1$, where $c_0$ is introduced in \eqref{def of c0}. Suppose $p$ satisfies \eqref{condition of p to be a solution} and $u\in \cH_p^{\gamma}(\tau)$. If $\alpha,\beta,\gamma,\nu,d,$ and $p$ satisfy
\begin{equation}
\label{condition of nu for holder embedding 1}
\frac{1}{\alpha p} < \mu < \frac{(\alpha(\nu+c_0/2)-\beta)\wedge1/2+1/2}{\alpha}\quad\text{and}\quad\frac{1}{\alpha p} < \nu <  \left( \frac{\gamma}{2} - \frac{d}{2p} \right)\wedge \left(1 - \frac{c_0}{2}\right),
\end{equation}
then  $u\in C^{\alpha\mu - 1/p}([0,\tau];C^{\gamma-2\nu-d/p})$ almost surely and
$$\bE\| u \|^p_{C^{\alpha\mu-1/p}\left([0,\tau];C^{\gamma-2\nu-d/p}(\bR^d)\right)} \leq N\| u \|^p_{\cH_p^{\gamma}(\tau)},
$$
where $N = N(\alpha,\beta,\gamma,d,p,T)$.

\end{corollary}
\begin{proof}
For the proof, we employ Lemma \ref{lemma:properties of bessel potential spaces} \eqref{item:Holder embedding for bessel potential space} and Theorem \ref{thm:embedding theorem for sol space}. Then, we have
\begin{equation*}
\bE\| u \|^p_{C^{\alpha\mu-1/p}\left([0,\tau];C^{\gamma-2\nu-d/p}(\bR^d)\right)} \leq \bE\| u \|^p_{C^{\alpha\mu-1/p}\left([0,\tau];H_p^{\gamma-2\nu}(\bR^d)\right)} \leq N\| u \|^p_{\cH_p^{\gamma}(\tau)}.
\end{equation*}
Thus, the corollary is proved.
\end{proof}

\vspace{2mm}

\section{Main Results}
\label{Main results}

This section presents the uniqueness, existence, $L_p$-regularity, and H\"older regularity of the solution to the following equation:
\begin{equation}
\label{eq:target equation}
\begin{aligned}
\partial_t^\alpha u &=  Lu + \bar b^i u u_{x^i} + \partial^{\beta}_t \sum_k \int_0^t  \sigma(u) \eta^k dw_s^k,\quad t>0;\quad u(0,\cdot) = u_0,
\end{aligned}
\end{equation}
where $Lu= a^{ij}u_{x^ix^j} + b^{i}u_{x^i} + cu$. The coefficients $a^{ij}$, $b^{i}$, and $c$ are $\cP\times \cB(\bR^d)$-measurable, $\bar b^i$ is $\cP\times \cB(\bR^{d-1})$-measurable, and $a^{ij}$, $b^i$, $c$, and $\bar b^i$ (and their derivatives) are uniformly bounded (see Assumption \ref{assumptions_on_coefficients_deterministic_part}). Additionally, we assume the coefficient $\bar b^i$ is independent of $x^i$. Indeed, because $\bar b^i$ is independent of $x^i$, we can employ the fundamental theorem of calculus to control the nonlinear term $\bar b^i u u_{x^i}$ (see Remark \ref{rmk: bar b is independent of xi}). Moreover, the diffusion coefficient $\sigma(u)$ is dominated by an $\bL_p$ function $h$ (see Assumption \ref{assumptions_on_coefficients_stochastic_part}) and it is used to obtain a uniform $L_p$ bound of the local solutions (see Remark \ref{rmk: sigma is bounded by h}).

In Theorem \ref{main theorem}, we obtain the existence and uniqueness of a solution in $\cH_p^{\gamma}$, where $\gamma\in (0,2-c_0-d/2)\wedge 1$. The components of equation \eqref{eq:target equation} affect the properties of the solution $u$. For example, if $\alpha$, $\beta$, $d$, and $p$ are given, the regularity $\gamma$ is determined. Remarks \ref{eq component explanations 0}, \ref{eq component explanations 1}, and \ref{eq component explanations 2} provide explanations for these relations.

Additionally, we have the maximal H\"older regularity of the solution in Corollary \ref{maximal holder regularity}. We employ the H\"older embedding theorem for solution spaces (Corollary \ref{corollary:continuous embedding of solution}). Furthermore, depending on the range of $\beta$, the behavior of the H\"older regularity of the solution in time varies. In detail, when $\beta \geq 1/2$, then the H\"older regularity of the solution in space is $\alpha/2$ times of the H\"older regularity of the solution in time. Moreover, if we consider the case $\alpha,\beta \uparrow 1$, then the H\"older regularity in time and space approaches $1/4$ and $1/2$, which are the results of the SPDEs driven by space-time white noise (e.g. \cite[Remark 8.7]{kry1999analytic} or \cite[Corollary 3.1]{han2022regularity}). In the case of $\beta < 1/2$, $1/2-\beta$ of the H\"older regularity in time is obtained due to the regularity of the stochastic integral (Remark \ref{regularity comparision}).

\vspace{2mm}

The following are assumptions on coefficients.

\begin{assumption} 
\label{assumptions_on_coefficients_deterministic_part}
\begin{enumerate}[(i)]

\item 
The coefficients $a^{ij} = a^{ij}(t,x)$, $b^i = b^i(t,x)$, and $c = c(t,x)$ are $\cP\times\cB(\bR^d)$-measurable.

\item 
The coefficient $\bar b^i(t,\bar x^i) = \bar b^i(t,x^1,\dots,x^{i-1},x^{i+1},\dots,x^d)$ is $\cP\times\cB(\bR^{d-1})$-measurable.

\item 
There exists $K>0$ such that 
\begin{equation}
\label{ellipticity_of_leading_coefficients_white_noise_in_time} 
K^{-1}|{\xi}|^2 \leq a^{ij}(t,x)\xi^i\xi^j \leq  K|{\xi}|^2\quad \text{for all}\quad (\omega,t,x)\in\Omega\times[0,\infty)\times\bR^d, \quad \xi \in \bR^d, 
\end{equation}
and
\begin{equation} 
\label{boundedness_of_deterministic_coefficients_white_noise_in_time} 
\sum_{i,j}\left| a^{ij}(t,\cdot) \right|_{C^{2}(\bR^d)} + \sum_{i}\left| b^{i}(t,\cdot) \right|_{C^{2}(\bR^d)} + |c(t,\cdot)|_{C^{2}(\bR^d)} + \sum_{i} \left|\bar b^i(t,\cdot)\right|_{C^{2}(\bR^{d-1})} \leq K
\end{equation}
for all $(\omega,t)\in\Omega\times[0,\infty)$.
\end{enumerate}
\end{assumption}

\begin{remark}
\label{rmk: bar b is independent of xi}

To prove the existence of a global solution, we need to acquire a uniform $L_p$ bound of the local solutions. Thus, we separate the local solutions into two parts: noise-dominating and nonlinear-dominating parts. In this remark, we consider the nonlinear-dominating parts related to $\bar b^i uu_{x^i}$. 

If coefficient $\bar b^i$ is independent of $x^i$, coefficient $\bar b^i$ can be taken out of the integral for $x^i$. Then, by the fundamental theorem of calculus to $x^i$, the nonlinear term $\bar b^i uu_{x^i}$ is eliminated in the $L_p$ estimate of the nonlinear-dominating part of the local solutions. Thus, the nonlinear-dominating parts are controlled by the initial data and diffusion coefficient $\sigma(u)$ (for more information, see Lemma \ref{lemma:nonlinear control}).

\end{remark}

To introduce the assumptions on the diffusion coefficient, we may assume $p\geq 2$.

\begin{assumption}[$p$]
\label{assumptions_on_coefficients_stochastic_part}
\begin{enumerate}[(i)]

\item 
The coefficient $\sigma(t,x,u)$ is $\cP\times\cB(\bR^d)\times\cB(\bR)$-measurable.

\item 
There exists a constant $K$ such that 
\begin{equation*}
|\sigma(t,x,u) - \sigma(t,x,v)| \leq K |u-v| \quad \text{for all}\quad (\omega,t,x)\in\Omega\times[0,\infty)\times\bR^d,\quad u,v\in\bR.
\end{equation*}

\item 
\label{function bound of the diffusion coefficient}
There exists a $\cP\times\cB(\bR^d)$-measurable function $h\in \bL_p$ such that
\begin{equation}
\label{bound of the stochastic coefficient}
|\sigma(t,x,u)| \leq |h(t,x)|\quad \text{for all}\quad (\omega,t,x)\in\Omega\times[0,\infty)\times\bR^d,\quad u\in\bR.
\end{equation}
\end{enumerate}
\end{assumption}

\begin{remark}
\label{rmk: sigma is bounded by h}

As mentioned in Remark \ref{rmk: bar b is independent of xi}, we divide the local solutions into two parts, and the nonlinear-dominating parts are controlled by the initial data $u_0$ and diffusion coefficients $\sigma(u)$. Then, to deal with the noise-dominating term and the terms including $\sigma(u)$, we employ the function $h(t,x)$ introduced in Assumption \ref{assumptions_on_coefficients_stochastic_part} $(p)$ \eqref{function bound of the diffusion coefficient}. Indeed, the terms related to the diffusion coefficient $\sigma(u)$ are controlled by the initial data and $h$ so that a uniform $L_p$ bound of $u$ is obtained (see Lemmas \ref{lemma:noise dominant part} and \ref{lemma:nonlinear control}).
\end{remark}

Next, we introduce the main results.

\vspace{2mm}

\begin{theorem}
\label{main theorem}
Let 
\begin{equation}
\label{condi alpha,beta,d,gamma}
\alpha\in (0,1), \quad \beta < \frac{3}{4}\alpha + \frac{1}{2}, \quad d<4 - 2c_0,\quad 0 < \gamma < (2-c_0-d/2)\wedge1
\end{equation}
and
\begin{equation}
\label{conditions of p in theorem}
p = 2^k\quad\text{for some}\quad k\in\bN\quad\text{and}\quad p>2\vee\frac{1}{\alpha}\vee \frac{1}{\alpha-\beta+1/2}\vee\frac{2+\alpha d}{\alpha\gamma} \vee \frac{d}{1-\gamma},
\end{equation}
where $c_0$ are the constants introduced in \eqref{def of c0}.
Suppose Assumptions \ref{assumptions_on_coefficients_deterministic_part} and \ref{assumptions_on_coefficients_stochastic_part} ($p$) hold.
If $u_0\in U_p^{\gamma}$, then there exists a unique solution $u\in \cH_{p,loc}^{\gamma}$ satisfying \eqref{eq:target equation}. Furthermore, for $\mu$ and $\nu$ satisfying \eqref{condition of nu for holder embedding 1}, and bounded stopping time $\tau\leq T$, we have
\begin{equation}
\label{time and space holder regularity}
u\in C^{\alpha\mu - 1/p}([0,\tau];C^{\gamma-2\nu-d/p})\quad \text{and} \quad\| u \|_{C^{\alpha\mu - \frac{1}{p}}\left([0,\tau];C^{\gamma-2\nu - d/p}\right)}  < \infty
\end{equation}
almost surely.
\end{theorem}
\begin{proof}
See {\bf{Proof of Theorem \ref{main theorem}}} in Section \ref{section:proof of main thm}.
\end{proof}

\begin{remark}
\label{eq component explanations 0}

\begin{enumerate}[(i)]

\item 
\label{condition of alpha}
We assume 
\begin{equation*}
\alpha \in (0,1)
\end{equation*}
because an inequality acting like the chain rule is employed to deal with the nonlinear-dominating part of the local solution (see Lemma \ref{lemma: chain ineq}).

\item 
The conditions 
\begin{equation*}
\beta < 3\alpha/4 + 1/2\quad\text{and}\quad d < 4-2c_0
\end{equation*}
are expected to obtain the uniqueness and existence of solutions to SPDEs with Caputo time fractional derivatives and space-time white noise even for the semilinear case. For example, see \cite[Section 7]{kim2019sobolev}. Additionally, observe that the choice of $\alpha$ and $\beta$ allows $d = 1,2,3$, where $c_0$ is the constant introduced in \eqref{def of c0}.
\end{enumerate}
\end{remark}

\begin{remark}
\label{eq component explanations 1}
\begin{enumerate}[(i)]

\item 
For the existence and uniqueness of local solutions, we impose
\begin{equation}
\label{regularity bound 1}
\gamma\in (0,2-c_0-d/2).
\end{equation}
Heuristically, if $u$ is a measurable, continuous, and bounded solution to equation \eqref{eq:target equation}, then for given $T<\infty$, we can define a bounded stopping time as follows:
\begin{equation*}
\tau_m := \inf\left\{ t\geq 0: \sup_{x\in\bR^d}|u(t,x)| \geq m  \right\}\wedge T.
\end{equation*}
Then, the solution $u$ satisfies the localized version of equation \eqref{eq:target equation} on $(0,\tau_m)$. In other words, 
\begin{equation}
\label{localized target equation}
\partial_t^\alpha u =  Lu + \frac{1}{2}\bar b^i \left( (|u|\wedge m)^2 \right)_{x^i} + \partial^{\beta}_t \sum_k \int_0^t  \sigma(u) \eta^k dw_s^k
\end{equation}
holds on $0<t<\tau_m$ with $u(0,\cdot) = u_0$. Then, as \eqref{localized target equation} is a semilinear equation, \eqref{regularity bound 1} has to be satisfied by \cite[Theorem 7.1]{kim2019sobolev} (for more detail, see \cite[Section 7]{kim2019sobolev} and \cite[Section 5]{kim2020sobolev}.

\item 
The following condition
\begin{equation}
\label{regularity bound 2}
\gamma\in(0,1)
\end{equation}
is assumed due to the nonlinear term $\bar b^i uu_{x^i}$ lowering the regularity of the solution. Even for SBEs ($\alpha = \beta = 1$), the condition in \eqref{regularity bound 2} is required (for more information, see \cite{gyongy1998existence,gyongy1999stochastic,han2022lp,han2022regularity,lewis2018stochastic}).

\end{enumerate}
\end{remark}

\begin{remark} 
\label{eq component explanations 2}
\begin{enumerate}[(i)]

\item 

To obtain the local solution, we employ the $L_p$ theory for the semilinear equation (see \cite[Theorem 5.1]{kry1999analytic}). When we control the nonlinear term $\bar b^i uu_{x^i}$ in the $L_p$ estimate, the kernel of $(1-\Delta)^{-\frac{\gamma-1}{2}}$ has to be controlled. Hence, 
\begin{equation*}
p>\frac{d}{1-\gamma}
\end{equation*}
is imposed (see Lemma \ref{lemma:existence uniqueness reg of local solution}).

\item 
We require $\bR$-valued continuous solutions to consider the cut-off version of equation \eqref{eq:target equation}. Therefore, we assume
\begin{equation*}
 p > 2\wedge \frac{1}{\alpha}\wedge\frac{1}{\alpha-\beta+1/2}\wedge \frac{2 + \alpha d}{\alpha\gamma}
\end{equation*}
which is required to apply the H\"older embedding theorem for $\cH_p^\gamma$ (see Theorem \ref{thm:embedding theorem for sol space} and Corollary \ref{corollary:continuous embedding of solution}).

\item 
As mentioned in Remark \ref{eq component explanations 0} \eqref{condition of alpha}, we employ an inequality similar to the chain rule. To apply \eqref{integration by part} instead of chain rule for the Caputo fractional derivative, we assume
\begin{equation*}
p = 2^k
\end{equation*}
for some $k\in\bN$.
\end{enumerate}

\end{remark}

To achieve the maximal H\"older regularity, we require the uniqueness of the solution in $p$.

\begin{theorem}
\label{uniqueness_in_p}
Suppose all the conditions of Theorem \ref{main theorem} hold. Let $u\in \cH_{p,loc}^{\gamma}$ be the solution of equation \eqref{eq:target equation} introduced in Theorem \ref{main theorem}. If $q > p$, $u_0\in U_{q}^{\alpha,\gamma}$, and Assumption \ref{assumptions_on_coefficients_stochastic_part} ($q$) hold, then $u\in \cH_{q,loc}^{\gamma}$.
\end{theorem}
\begin{proof}
See {\bf{Proof of Theorem \ref{uniqueness_in_p}}} in Section \ref{section:proof of main thm}.
\end{proof}

Finally, we obtain the maximal H\"older regularity of the solution by combining Theorems \ref{main theorem} and \ref{uniqueness_in_p}. Recall that $c_0$ is introduced in \eqref{def of c0}.

\begin{corollary}
\label{maximal holder regularity}
Suppose $\alpha$, $\beta$, $d$, and $\gamma$ satisfy \eqref{condi alpha,beta,d,gamma}, and  $u_0\in \cap_{p>2}U_p^{(2-c_0-d/2)\wedge1}$, and $h\in \cap_{p>2}\bL_p$ satisfies \eqref{bound of the stochastic coefficient}. Then, for any bounded stopping time $\tau \leq T$, \eqref{intro: holder reg of sol} holds almost surely. 
\end{corollary}
\begin{proof}
When $\alpha$, $\beta$, $d$, and $\gamma$ are given in \eqref{condi alpha,beta,d,gamma}, we choose $p$ as in \eqref{conditions of p in theorem}. For each $p$, there exists a unique solution $u_p\in\cH^\gamma_{p,loc}$ to equation \eqref{eq:target equation}. Due to Theorem \ref{uniqueness_in_p}, $u_p\in \cH^{\gamma}_{q,loc}$ for any $q \geq p$ so that we write $u$ instead of $u_p$ and $u$ is independent of $p$. Thus, by letting $p$ large in \eqref{time and space holder regularity}, we have \eqref{intro: holder reg of sol}. Thus, the corollary is proved.
\end{proof}

\begin{remark}
\label{regularity comparision}
\begin{enumerate}[(i)]

\item 
If $1/2 \leq \beta <\alpha + 1/2$, the H\"older regularity in space is $\alpha/2$ times that in time. Furthermore, we can recover the H\"older regularity results of SBEs ($\alpha = \beta = 1$) by considering the case $\alpha,\beta\uparrow 1$. We cite \cite[Proposition 5.1]{lewis2018stochastic} or \cite[Corollary 3.1]{han2022regularity} for reader's convenience.

\item  
If $\beta < 1/2$, then the H\"older regularity in time obtains additional regularity by as much as $1/2-\beta$. This phenomenon is caused by the stochastic integral of equation \eqref{eq:target equation} adding the H\"older regularity of noise in time almost $1/2$, and $\partial_t^\beta$ reducing the regularity of the noise by $\beta$.

\end{enumerate}
\end{remark}

\vspace{2mm}

\section{Proof of Theorems \ref{main theorem} and \ref{uniqueness_in_p}}
\label{section:proof of main thm}

We assume that all conditions in Theorem \ref{main theorem} hold for the remainder of this section.

To establish the existence of a global solution, we need to obtain the uniqueness and existence of local solutions (Lemma \ref{lemma:existence uniqueness reg of local solution}). With these local solutions, we build a candidate for a global solution. More precisely, we paste the local solutions and demonstrate that the local existence time explodes almost surely (Lemma \ref{lemma:key_lemma}). To prove that the local existence time explodes almost surely, we demonstrate that a uniform $L_p$ bound of local solutions exists. In detail, we separate the local solution into noise- and nonlinear-dominating parts. The noise-dominating part is affected by the stochastic part of the equation, and the other part is influenced by the nonlinear term $b^i uu_{x^i}$. When we deal with the noise-dominating part of the solution, the dominating function of the diffusion coefficient provides a uniform $L_p$ bound for the noise-dominating part of the local solutions (see Assumption \ref{assumptions_on_coefficients_stochastic_part} (p) \eqref{function bound of the diffusion coefficient} and Lemma \ref{lemma:noise dominant part}). The other part is controlled by employing a version of the chain rule and Gr\"onwall inequality (see Lemmas \ref{lemma: chain ineq} and \ref{lemma:nonlinear control} and Theorem \ref{time fractional integral gronwall ineq}).

\vspace{2mm}

First, we introduce the uniqueness and existence theorem for semilinear SPDEs.

\begin{assumption}[$\tau$]
\label{assumptions on f and g}
\begin{enumerate}[(i)]
\item 
\label{condi for semi linear 1}
The functions $f(t,x,u)$ and $g^k(t,x,u)$ are $\cP\times\cB(\bR^d)\times\cB(\bR)$-measurable functions satisfying the following:
\begin{equation*}
f(t,x,0)\in\bH_p^\gamma(\tau) \quad\text{and}\quad g(t,x,0) = (g^1(t,x,0),g^2(t,x,0),\dots)\in\bH_p^{\gamma+1}(\tau,l_2).
\end{equation*}

\item 
\label{condi for semi linear 2}
For any $\ep>0$, there exists a constant $N_\ep$ such that for any $u,v\in\bH_p^{\gamma}(\tau)$,
\begin{equation*}
\| f(u) - f(v) \|^p_{\bH_p^{\gamma-2}(\tau)} + \| g(u) - g(v) \|^p_{\bH_p^{\gamma-2+c_0}(\tau,l_2)} \leq \ep\|u-v\|^p_{\bH_p^{\gamma}(\tau)} + N_\ep\|u-v\|^p_{\bH_p^{\gamma-2}(\tau)},
\end{equation*}
where $c_0$ is the constant introduced in \eqref{def of c0}.

\end{enumerate}
\end{assumption}

\begin{lemma}
\label{lemma:quasi-linear results}
Let $\tau \leq T$ be a bounded stopping time. Suppose Assumption \ref{assumptions on f and g} $(\tau)$ hold. Then, for initial data $u_0\in U_p^{\gamma}$, the following equation:
\begin{equation}
\label{quasi-linear_equation}
\partial_t^\alpha u = L u + f(u) + \partial_t^\beta \int_0^t g^k(u) dw_t^k,\quad 0<t\leq \tau;\quad u(0,\cdot) = u_0
\end{equation}
has a unique solution $u\in\cH_p^{\gamma}(\tau)$. Moreover,
\begin{equation}
\label{quasi-linear_estimate}
\| u \|^p_{\cH_p^{\gamma}(\tau)} \leq N \left( \|u_0\|^p_{U_p^{\gamma}} + \| f(0) \|^p_{\bH_p^{\gamma-2}(\tau)} + \| g(0) \|^p_{\bH^{\gamma-2+c_0}_p(\tau,l_2)} \right),
\end{equation}
where $N = N(\alpha,\beta,\gamma,d,p,K,T)$ and $c_0$ is the constant introduced in \eqref{def of c0}.
\end{lemma}
\begin{proof}
Theorem 5.1 of \cite{kry1999analytic} is the motivation of the proof. The case $\tau \equiv T$ is obtained by \cite[Theorem 2.18]{kim2020sobolev}; thus, we only consider the case $\tau \leq T$. 

\textit{(Step 1). (Existence)}
 Set
$$
\bar{f}(t, u) := 1_{t\leq \tau} f(t,u)\quad\text{and}\quad \bar{g}(t, u) := 1_{t\leq \tau} g(t,u).
$$
Additionally, $\bar{f}(u)$ and $\bar{g}(u)$ satisfy Assumption \ref{assumptions on f and g} ($T$). Then, by \cite[Theorem 2.18]{kim2020sobolev}, there exists a unique solution $u\in \cH_{p}^{\gamma}(T)$ such that $u$ satisfies equation \eqref{quasi-linear_equation} with $\bar f$ and $\bar{g}$, instead of $f$ and $g$, respectively. As $\tau\leq T$, we have $u\in\cH_{p}^{\gamma}(\tau)$ and $u$ satisfies equation \eqref{quasi-linear_equation} and estimate \eqref{quasi-linear_estimate} with $f$ and $g$.

\textit{(Step 2). (Uniqueness)}
Let $u,v\in \cH_{p}^{\gamma}(\tau)$ be two solutions of equation \eqref{quasi-linear_equation}. Then, \cite[Theorem 2.18]{kim2020sobolev} yields there exists a unique solution $\bar{v}\in \cH_{p}^{\gamma}(T)$ satisfying
\begin{equation}
\label{equation_in_proof_of_uniqueness}
\partial_t^\alpha \bar{v}=L\bar{v} + \bar{f}(v)  + \sum_{k=1}^{\infty} \partial^{\beta}_t\int_0^t\bar{g}^k(v)  dw^k_t, \quad 0<t\leq T\, ; \quad \bar{v}(0,\cdot)=u_0.
\end{equation}
Notice that in \eqref{equation_in_proof_of_uniqueness}, $\bar{f}(v)$ and $\bar{g}(v)$ are used instead of $\bar{f}(\bar{v})$ and $\bar{g}(\bar{v})$, respectively.  Set  $\tilde v:=v-\bar{v}$. Then, for fixed $\omega\in\Omega$, we have
$$
\partial_t^\alpha \tilde v = L  \tilde v , \quad 0 < t \leq \tau\, ; \quad \tilde v(0,\cdot)=0.
$$
By the deterministic version of \cite[Theorem 2.18]{kim2020sobolev}, we have $\tilde v = 0$ in $L_p((0,\tau]\times\bR^d)$ almost surely. Additionally, it implies $v(t,\cdot) = \bar{v}(t,\cdot)$ in $L_p((0,\tau]\times\bR^d)$ almost surely. Thus, in equation \eqref{equation_in_proof_of_uniqueness}, we can replace $\bar f (v)$ and $\bar{g}(v)$ with $\bar f (\bar v)$ and $\bar{g}(\bar{v})$. Therefore, $\bar v\in \cH_p^{\gamma}(T)$ satisfies equation \eqref{quasi-linear_equation} on $(0,T]$ with $\bar f, \bar g$ instead of $f,g$, respectively. Similarly, by following word for word, there exists $\bar{u}\in \cH_p^{\gamma}(T)$ such that $\bar{u}$ satisfies equation \eqref{quasi-linear_equation} on $(0,T]$ with $\bar f$ and $\bar g$ instead of $f$ and $g$.
Thus, by the uniqueness result in $\cH_{p}^{\gamma}(T)$, we have $\bar{u} = \bar{v}$ in $\cH_{p}^{\gamma}(T)$, which implies $u = v$ in $\cH_p^{\gamma}(\tau)$. Thus, the lemma is proved.
\end{proof}

Next, we provide the uniqueness and existence of a local solution to equation \eqref{eq:target equation}. As an auxiliary function, we choose $\rho(\cdot)\in C_c^\infty(\bR)$ such that $\rho(z)\geq0$ on $z\in(-\infty,\infty)$, $\rho(z) = 1$ on $|z|\leq 1$, $\rho(z) = 0$ on $|z|\geq2$, and $\frac{d}{dz}\rho(z)\leq 0$ on $z\geq0$. We define the following:
\begin{equation}
\label{def of cut off function}
\rho_m(z) := \rho(z/m).
\end{equation}

\begin{lemma}
\label{lemma:existence uniqueness reg of local solution}
Let $\tau \leq T$ be a bounded stopping time. For $m\in\bN$, there exists $u_m\in \cH_p^{\gamma}(\tau)$ such that
\begin{equation*}
\partial_t^\alpha u_m = L u_m + \bar b^i \left( u_m^{2}\rho_m(u_m) \right)_{x^i} + \partial_t^\beta\int_0^t \sigma(t,x,u_m)\eta^k(x) dw^k_t,\,\, 0<t\leq \tau;\,\,u_m(0,\cdot) = u_0,
\end{equation*}
where $\rho_m$ is the function introduced in \eqref{def of cut off function}. Furthermore, $u_m\in C([0,\tau];C(\bR^d))$ almost surely and
\begin{equation}
\label{continuity of local solutions}
\bE\sup_{t\leq\tau}\sup_{x\in\bR^d}| u_m(t,x) |^p \leq N \| u_m \|_{\cH_p^\gamma(\tau)}^p < \infty
\end{equation}
almost surely.
\end{lemma}
\begin{proof}
Due to Lemma \ref{lemma:quasi-linear results} and Corollary \ref{corollary:continuous embedding of solution}, it suffices to show that Assumption \ref{assumptions on f and g} $(\tau)$ holds. Because $\sigma(t,x,0) \leq h(t,x)$ for all $\omega,t,x$ and $h\in \bL_p$, Assumption \ref{assumptions on f and g} \eqref{condi for semi linear 1} is satisfied.

In the case of Assumption \ref{assumptions on f and g} \eqref{condi for semi linear 2}, notice that for $u,v\in \bR$, we have
\begin{equation*}
\left| u^{2}\rho_m(u) - v^{2}\rho_m(v) \right| \leq N_m|u-v|.
\end{equation*}
Then, for $u,v\in\bH_p^\gamma(\tau)$, by Remark \ref{Kernel} and Lemmas \ref{lemma:properties of bessel potential spaces} \eqref{pointwise_multiplier} and \eqref{bounded_operator}, we have
\begin{equation}
\label{estimate for local sol 1}
\begin{aligned}
&\left\|\bar b^i \left((u(t,\cdot))^{2} \rho_m(u(t,\cdot)) - (v(t,\cdot))^{2} \rho_m(v(t,\cdot))\right)_{x^i} \right\|^p_{H_p^{\gamma-2}} \\
& \leq N\left\| (u(t,\cdot))^{2} \rho_m(u(t,\cdot)) - (v(t,\cdot))^{2} \rho_m(v(t,\cdot)) \right\|^p_{H^{\gamma-1}_p} \\
&\leq N\int_{\bR^d} \left( \int_{\bR^d} \left| R_{1-\gamma}(x-y) \right|\left( (u(\cdot))^{2} \rho_m(\cdot,u(\cdot)) - (v(\cdot))^{2} \rho_m(\cdot,v(\cdot)) \right)(t,y) dy \right)^p dx \\
& \leq N_m\left( \int_{\bR^d}|R_{1-\gamma}(x)|dx \right)^p \int_{\bR^d} |u(t,x) - v(t,x)|^p dx \\
\end{aligned}
\end{equation}
and
\begin{equation} \label{estimate for local sol 2}
\begin{aligned}
&\left\| \sigma(u)\bm{\eta} - \sigma(v)\bm{\eta} \right\|_{H_p^{\gamma-2+c_0}(l_2)}^p \\
& \leq \int_{\bR^d}  \left( \sum_k \left( \int_{\bR^d} \left| R_{-\gamma+2-c_0}(x-y) \right|(\sigma(\cdot,u(\cdot)) - \sigma(\cdot,v(\cdot)))(t,y)\eta^k(y) dy \right)^2 \right)^{p/2} dx \\
& \leq \int_{\bR^d}  \left( \int_{\bR^d} \left|R_{-\gamma+2-c_0}(x-y)\right|^2\left( \sigma(t,y,u(t,y)) - \sigma(t,y,v(t,y)) \right)^2 dy \right)^{p/2} dx \\
& \leq K^p \int_{\bR^d}  \left( \int_{\bR^d} \left|R_{-\gamma+2-c_0}(y)\right|^2(u(t,x-y) - v(t,x-y))^2 dy \right)^{p/2} dx \\
& \leq K^p\left( \int_{\bR^d} \left| R_{-\gamma+2-c_0}(y) \right|^2 dy \right)^{p/2}\int_{\bR^d}  |u(t,x) - v(t,x)|^p  dx \\
\end{aligned}
\end{equation}
on almost every $(\omega,t)\in\opar0,\tau\cbrk$. Due to Remark \ref{Kernel}, we have 
\begin{equation*}
\int_{\bR^d} \left|R_{1-\gamma}(y)\right| dy + \int_{\bR^d} \left| R_{-\gamma+2-c_0}(y) \right|^2 dy<\infty.
\end{equation*} 
By integrating with respect to $(\omega,t)$ to \eqref{estimate for local sol 1} and \eqref{estimate for local sol 2}, employing Lemma \ref{lemma:properties of bessel potential spaces} \eqref{multi_ineq}, and Young's inequality, we have
\begin{equation}
\label{estimate for local sol 3}
\begin{aligned}
&\left\| \bar b^i \left(u^2 \rho_m(u) - v^{2} \rho_m(v)\right)_{x^i} \right\|_{\bH_p^{\gamma-2}(\tau)}^p + \left\| \sigma(u)\eta - \sigma(v)\eta \right\|_{\bH_p^{\gamma-2+c_0}(\tau,l_2)}^p\\
&\quad \leq  N_m\|u-v\|_{\bL_p(\tau)}^p \\
&\quad \leq  \ep\|u-v\|_{\bH^{\gamma}_p(\tau)}^p + N_m\|u-v\|_{\bH^{\gamma-2}_p(\tau)}^p.
\end{aligned}
\end{equation}

To obtain the second assertion, we employ Corollary \ref{corollary:continuous embedding of solution}.
The lemma is proved.

\end{proof}

\begin{remark} \label{def_u}
We introduce a candidate for a global solution. Let $T<\infty$. For $m\in\bN$, let $u_m\in \cH_p^{\gamma}(T)$ be the solution introduced in Lemma \ref{lemma:existence uniqueness reg of local solution}.
Then, for $R\in\{ 1,2,\dots,m \}$, define a stopping time $\tau_m^R$
\begin{equation} 
\label{stopping_time_taumm}
\tau_m^R := \inf\left\{ t\geq0 : \sup_{x\in\bR} |u_m(t,x)|\geq R \right\}\wedge T.
\end{equation}
Observe that 
\begin{equation}
\label{local_existence_time_increasing}
\tau_R^R \leq \tau_m^m
\end{equation}
Indeed, if $R = m$, \eqref{local_existence_time_increasing} is obvious. If $R<m$, we have $u_m\wedge m=u_m\wedge m\wedge R  = u_m \wedge R$ for $t\leq\tau_m^R$. Therefore, $u_m$ and $u_R$ are solutions to equation
\begin{equation*}
\partial_t^\alpha u = L u +  \bar{b}^i\left(u^{2}\rho_R(u) \right)_{x^i} + \sigma(u)\eta^k dw_t^k, \quad0<t\leq\tau_m^R\,; \quad u(0,\cdot)=u_0.
\end{equation*}
In contrast, $u_R\wedge R = u_R\wedge R\wedge m = u_R\wedge m$ for $t\leq \tau_R^R$. Thus, $u_m$ and $u_R$ are solutions to equation
\begin{equation*}
\partial_t^\alpha u = L u +  \bar{b}\left(u^{2}\rho_m(u) \right)_{x^i} + \sigma(u)\eta^k dw_t^k, \quad0<t\leq\tau_R^R\,; \quad u(0,\cdot)=u_0.
\end{equation*}
Observe that the uniqueness and continuity results in Lemma \ref{lemma:existence uniqueness reg of local solution} yields that $u_m=u_R$ for all $t \leq (\tau_m^R\vee\tau_R^R)$. Therefore, for $t \leq \tau_m^R$, 
\begin{equation*}
\sup_{s\leq t}\sup_{x\in\bR}|u_R(s,x)|=\sup_{s\leq t}\sup_{x\in\bR}|u_m(s,x)|\leq R,
\end{equation*}
and this implies  $\tau_m^R\leq\tau_R^R$. Similarly, $\tau_m^R\geq\tau_R^R$; thus, 
$$ \tau_R^R=\tau_m^R
$$ 
almost surely.
Moreover, we have $\tau_m^R \leq \tau_m^m$ because $m > R$. Therefore, we have \eqref{local_existence_time_increasing}.

Further, we define
\begin{equation*}
u(t,x) := u_{m}(t,x)\quad\text{on}\quad t\leq \tau_m^m
\end{equation*}
and set 
\begin{equation}
\label{def of tau infty}
\tau_\infty := \limsup_{m\to\infty}\limsup_{T\to\infty}\tau_m^m.
\end{equation}
It should be remarked that $u(t,x)$ is well-defined on $\Omega\times[0,\infty)\times\bR^d$ and the nontrivial domain of $u$ is $\Omega\times[0,\tau_\infty)\times\bR^d$.

\end{remark}

To obtain a uniform $L_p$ bound of the local solution $u_m$, we separate $u_m$ into noise- and nonlinear-dominating parts. Lemma \ref{lemma:noise dominant part} provides the existence, uniqueness, and estimate of noise-dominating parts of $u_m$.

\begin{lemma}
\label{lemma:noise dominant part}
Let $T<\infty$. Then, there exists $v\in \cH_p^{\gamma}(T)$ such that
\begin{equation*}
\partial_t^\alpha v = L v + \partial_t^\beta\int_0^t \sigma(s,x,u)\eta^k(x) dw^k_s,\quad 0<t\leq T,\quad u(0,\cdot) = u_0
\end{equation*}
Furthermore,  $v\in C([0,T];C(\bR^d))$ almost surely, and
\begin{equation*}
\bE\sup_{t\leq T}\sup_{x\in\bR^d}|v(t,x)|^p + \bE\sup_{t\leq T}\|v(t,\cdot)\|_{L_p}^p\leq N\| v \|_{\cH_p^{\gamma}(T)}^p \leq N\| u_0 \|_{U_p^{\gamma}}^p + N\| h \|_{\bL_p(\tau)}^p,
\end{equation*}
where $N = N(\alpha,\beta,\gamma,d,p,K,T)$.
\end{lemma}
\begin{proof}
Similar to the proof of Lemma \ref{lemma:existence uniqueness reg of local solution}, it is enough to show that Assumption \ref{assumptions on f and g} $(\tau)$ holds.
Set $\bm{\eta} = (\eta^1,\eta^2,\dots)$. Then, by Remark \ref{Kernel}, for $t\leq T$
\begin{equation}
\label{bound of stochastic part}
\begin{aligned}
&\|\sigma(t,\cdot,u(t,\cdot))\bm{\eta}\|^p_{H_p^{\gamma-2+c_0}(l_2)} \\
&= \int_{\bR^d} \left( \sum_{k=1}^\infty\left( \int_{\bR^d} R_{-\gamma+2-c_0}(x-y)\sigma(t,y,u(t,y))\eta^k(y) dy \right)^2 \right)^{p/2} dx \\
&= \left( \int_{\bR^d} |R_{-\gamma+2-c_0}(x)|^2 dx \right)^{p/2} \int_{\bR^d} |\sigma(t,y,u(t,y))|^p dy  \\
&\leq \left( \int_{\bR^d} |R_{-\gamma+2-c_0}(x)|^2 dx \right)^{p/2} \int_{\bR^d} |h(t,y)|^p dy  \\
&\leq N\| h(t,\cdot) \|_{L_p}^p.
\end{aligned}
\end{equation}
Therefore, 
\begin{equation*}
\begin{aligned}
\left\| \sigma(u)\bm{\eta} \right\|_{\bH_p^{\gamma - 2 + c_0}(T,l_2)}^p
& \leq \bE\int_0^T\| \sigma(t,\cdot,u(t,\cdot))\bm{\eta} \|_{H_p^{\gamma-2+c_0}(l_2)}^p dt
\leq N\| h \|_{\bL_p}^p.
\end{aligned}
\end{equation*}
Thus, the lemma is proved by Lemma \ref{lemma:quasi-linear results}.
\end{proof}

Next, we control the nonlinear-dominating parts of the local solutions. The following two lemmas are crucial in obtaining uniform $L_p$ bounds. Lemma \ref{lemma: chain ineq} functions as a chain rule, and Theorem \ref{time fractional integral gronwall ineq} is a version of the Gr\"onwall inequality.

\begin{lemma}
\label{lemma: chain ineq}
Suppose $\alpha\in(0,1)$ and $k\in\bN$. For any $\psi\in C_c^\infty((0,\infty)\times\bR^d)$, we have
\begin{equation}
\label{integration by part}
\partial_t^\alpha (\psi(\cdot,x))^{2^k}(t) \leq 2^{k} \psi(t,x)|\psi(t,x)|^{2^{k}-2} \partial_t^\alpha \psi(t,x),
\end{equation}
for all $(t,x)\in(0,\infty)\times\bR^d$.
\end{lemma}
\begin{proof}
We employ the mathematical induction. The results and proof are motivated by (4.2) of \cite{dong2019lp}. 

\textit{(Step 1).} 
First, we consider the case $k = 1$. Although the proof is in the proof of \cite[Proposition 4.1]{dong2019lp}, we include the proof for the completeness of this paper.

Let $\psi\in C_c^\infty((0,\infty)\times\bR^d)$ and $t\in(0,\infty)$ and $x\in\bR^d$. For $s\in (0,t]$, set
\begin{equation*}
F_1(s) := \frac{1}{2}|\psi(s,x)|^2,\quad F_2(s) := \psi(s,x)\psi(t,x),
\end{equation*}
and
\begin{equation*}
\begin{aligned}
F(s) 
&:= \frac{1}{2}\left( |\psi(s,x)|^2 - |\psi(t,x)|^2 \right) - (\psi(s,x) - \psi(t,x))\psi(t,x). \\
\end{aligned}
\end{equation*}
Further,
\begin{equation*}
F(s) = \frac{1}{2}|\psi(s,x) - \psi(t,x)|^2 \geq 0
\end{equation*}
on $s\leq t$, and the equality holds for $s = t$. Notice that the integration by parts implies that
\begin{equation*}
\begin{aligned}
\int_0^t (t-s)^{-\alpha}(F_1'(s) - F_2'(s))ds = \int_0^t (t-s)^{-\alpha} F'(s)ds\leq 0.
\end{aligned}
\end{equation*}
Then, by the definition of $\partial_t^\alpha$ (Definition \ref{def of time fractional derivative}), we have \eqref{integration by part} with $k = 1$.

\vspace{2mm}

\textit{(Step 2).}
Let $n\in\bN$ and assume that the results hold for $k = 1,2,\dots,n-1$. Set $\tilde \psi(t,x):=(\psi(t,x))^{2}$. Since $\tilde \psi(t,x)\in C_c^\infty((0,\infty)\times\bR^d)$, we have
\begin{equation*}
\begin{aligned}
\partial_t^\alpha (\psi(\cdot,x))^{2^n}(t)
& = \partial_t^\alpha (\tilde \psi(\cdot,x))^{2^{n-1}}(t) \\
& \leq 2^{n-1} \tilde \psi(t,x) \left| \tilde \psi(t,x) \right|^{2^{n-1}-2} \partial_t^\alpha \tilde \psi(t,x)\\
& = 2^{n-1}  \left| \psi(t,x) \right|^{2^{n}-2}\partial_t^\alpha (\psi(t,x))^2 \\
& \leq 2^n  \psi(t,x)\left|  \psi(t,x) \right|^{2^{n}-2}\partial_t^\alpha \psi(t,x).
\end{aligned}
\end{equation*}
The lemma is proved.

\end{proof}

\begin{theorem}[Theorem 8 of \cite{almeida2017gronwall}]
\label{time fractional integral gronwall ineq}
Let $\psi(t)$ be a nonnegative integrable function on $[0,T]$. For a constant $N_1$, if the function $\psi$ satisfies
$$ \psi(t) \leq \psi_0 + N_1I_t^\alpha \psi 
$$
on $t\in[0,T]$, then 
$$ \psi(t) \leq \left(1 + \sum_{k=0}^\infty  \frac{N_1^k }{\Gamma(k\alpha)}\frac{(\Gamma(\alpha)t^{\alpha})^k}{k\alpha}  \right)\psi_0
$$
on $t\in[0,T]$.
\end{theorem}

We consider following lemma to control the remainder of the local solution $u_m$.
\begin{lemma}
\label{lemma:nonlinear control}
Let $u_m\in\cH_p^\gamma(T)$ and $v\in\cH_p^\gamma(T)$ be functions introduced in Lemmas \ref{lemma:existence uniqueness reg of local solution} and \ref{lemma:noise dominant part}, and $\tau_m^m$ be the stopping time introduced in \eqref{stopping_time_taumm}. Then, 
\begin{equation*}
\begin{aligned}
&\| u_m(t,\cdot) - v(t,\cdot) \|_{L_p}^p \\
& \leq N \sup_{t\leq T}\sup_{x\in\bR^d}|v(t,x)|^p\sup_{t\leq T}\| v(s,\cdot) \|_{L_{p}}^{p} \left[ 1+ \sum_{k=0}^\infty\frac{\left(  1+\sup_{s\leq t,x\in\bR^d}|v(s,x)|^2 \right)^k }{\Gamma(k\alpha)} \frac{(\Gamma(\alpha)T^{\alpha})^k}{k\alpha}\right]
\end{aligned}
\end{equation*}
for all $t\leq \tau_m^m$ almost surely, where $N = N(p,K)$.
\end{lemma}
\begin{proof}
Set 
\begin{equation*}
w_m := u_m - v\quad\text{and}\quad f_m:=L w_m + \bar b^i((u_m)^2 \rho_m(u_m))_{x^i}.
\end{equation*}
Then, we have $f_m\in\bH_p^{\gamma-2}(T)$ since $w_m,u_m\in \bH_p^{\gamma}(T)$ and estimates similar to \eqref{estimate for local sol 3}. Additionally, $\partial_t^\alpha w_m = f_m$. Let $(\omega,t)\in\opar0,\tau_m^m\cbrk$. Due to \cite[Remark 2.9]{kim2019sobolev}, there exists $w_m^n\in C_c^\infty((0,\infty)\times\bR^d)$ such that $w_m^n\to w_m$ in $L_p((0,t); H_p^{\gamma})$, and $\partial_t^\alpha w_m^n$ is a Cauchy in $L_p((0,t); H_p^{\gamma-2})$. Define 
$$f_m^n := \partial_t^\alpha w_m^n.
$$ 
Moreover, $f_m$ is the limit of $f_m^n$ as $n\to\infty$ in $L_p((0,t); H_p^{\gamma-2})$ (see \cite[Remark 2.9]{kim2019sobolev}).  Choose a nonnegative function $\zeta\in C_c^\infty(\bR^d)$ with a unit integral and set $\zeta_\ep(x) := \ep^{-d}\zeta(x/\ep)$ for $\ep>0$. For $h\in L_{1,loc}(\bR^d)$, set $h^{(\ep)}(x) := \int_{\bR^d} h(y)\zeta_{\ep}(x-y)dy.$

Next, let $\ep>0$ and $x\in\bR^d$. Since $w_m^{n(\ep)}\in C_c^\infty((0,\infty)\times\bR^d)$ and $p = 2^k$, Lemma \ref{lemma: chain ineq} yields
\begin{equation}
\label{apply iteration lemma}
\begin{aligned}
\frac{1}{p}\partial_t^\alpha \left(w_m^{n(\ep)}(\cdot,x)\right)^p(t) 
&\leq  f_m^{n(\ep)}  (t,x) w_m^{n(\ep)}(t,x) \left|w_m^{n(\ep)}(t,x) \right|^{p-2}
\end{aligned}
\end{equation}
on $t\in(0,\infty)$.
Additionally, as $w_m^{n}(0,x) = 0$, we have 
\begin{equation}
\label{initial data 0}
w_m^{n(\ep)}(0,x) = 0.
\end{equation}
Thus, if we take stochastic integral $I_t^\alpha$ on both sides of \eqref{apply iteration lemma}, we have
\begin{equation}
\label{remainder part integral form}
\begin{aligned}
\frac{1}{p} \left|w_m^{n(\ep)}(t,x)\right|^p 
\leq I_t^{\alpha}\left[ f_m^{n(\ep)}(\cdot,x)  w_m^{n(\ep)}(\cdot,x) \left| w_m^{n(\ep)}(\cdot,x) \right|^{p-2} \right]
\end{aligned}
\end{equation}
due to $\left( w_m^{n(\ep)} \right)^p\in C_c^\infty((0,\infty)\times\bR^d)$, \eqref{initial data 0}, and Remark \ref{rmk:prop of fractional calculus}.
Observe that \eqref{ineq:lp ineq for time fractional int} with $q = \infty$ and the H\"older inequality imply that
\begin{equation}
\label{approximation of integral form of remainder part 1}
\begin{aligned}
&\left\| I^\alpha_{\cdot} \left[  f_m^{n(\ep)}(\cdot,x)w^{n(\ep)}_m(\cdot,x) \left| w^{n(\ep)}_m(\cdot,x) \right|^{p-2} - f_m^{(\ep)}(\cdot,x) w^{(\ep)}_m(\cdot,x)\left| w^{(\ep)}_m(\cdot,x) \right|^{p-2} \right] \right\|_{L_1((0,t))} \\
& \leq \int_0^t \left|  f_m^{n(\ep)}(s,x) w^{n(\ep)}_m(s,x)\left| w^{n(\ep)}_m(s,x) \right|^{p-2}  - f_m^{(\ep)}(s,x) w^{(\ep)}_m(s,x) \left| w^{(\ep)}_m(s,x) \right|^{p-2}   \right| ds \\
& \leq N\int_0^t \left|     f_m^{n(\ep)}(s,x)  - f_m^{(\ep)}(s,x) \right| \left| w^{n(\ep)}_m(s,x) \right|^{p-1} ds \\
&\quad+ N\int_0^t\left|f_m^{(\ep)}(s,x) \left[ w^{n(\ep)}_m(s,x)\left| w^{n(\ep)}_m(s,x) \right|^{p-2} -  w^{(\ep)}_m(s,x)\left| w^{(\ep)}_m(s,x) \right|^{p-2} \right]  \right| ds \\
& \leq N \bigg[A_n\left\| w^{n(\ep)}_m(\cdot,x) \right\|_{L_p(0,t)}^{2} 
+  B_n C_n\left\| f_m^{(\ep)}(\cdot,x) \right\|_{L_p(0,t)} \bigg] \left\| w_m^{n(\ep)}(\cdot,x) \right\|_{L_p(0,t)}^{p-3},  \\
\end{aligned}
\end{equation}
where 
\begin{equation*}
\begin{aligned}
A_n = \left\| f_m^{n(\ep)}(\cdot,x)  - f_m^{(\ep)}(\cdot,x) \right\|_{L_p(0,t)} \quad
B_n = \left\| w_m^{n(\ep)}(\cdot,x) - w_m^{(\ep)}(\cdot,x) \right\|_{L_p(0,t)},
\end{aligned}
\end{equation*}
and
$$ C_n = \left\| w_m^{n(\ep)}(\cdot,x) \right\|_{L_p(0,t)} + \left\| w_m^{(\ep)}(\cdot,x) \right\|_{L_p(0,t)}.
$$
Moreover,
\begin{equation}
\label{approximation of integral form of remainder part 2}
A_n,B_n\to0\quad\text{and}\quad C_n \to 2  \left\| w_m^{(\ep)}(\cdot,x) \right\|_{L_p(0,t)}\quad\text{as}\quad n\to \infty
\end{equation}
since $w_m^{n}\to w_m$ and $f_m^n \to f_m$ in $L_p((0,t);H_p^\gamma)$. Then, by applying \eqref{approximation of integral form of remainder part 2} to \eqref{approximation of integral form of remainder part 1}, we have
\begin{equation*}
\left\| I^\alpha_{\cdot} \left[  f_m^{n(\ep)}(\cdot,x)w^{n(\ep)}_m(\cdot,x) \left| w^{n(\ep)}_m(\cdot,x) \right|^{p-2} - f_m^{(\ep)}(\cdot,x) w^{(\ep)}_m(\cdot,x)\left| w^{(\ep)}_m(\cdot,x) \right|^{p-2} \right] \right\|_{L_1((0,t))} \to 0
\end{equation*}
as $n\to\infty$. Therefore, there exists a sequence $n_l$ such that $w^{n_l(\ep)}_m(\cdot,x) \to  w^{(\ep)}_m(\cdot,x)$ and $I_\cdot^\alpha \left[ f_m^{n_l(\ep)} w^{n_l(\ep)}_m \left| w^{n_l(\ep)}_m \right|^{p-2} \right] \to I_\cdot^\alpha \left[ f_m^{(\ep)} w^{(\ep)}_m \left| w^{(\ep)}_m \right|^{p-2} \right]$ almost everywhere on $[0,t]$. Furthermore, the convergence holds everywhere on $[0,t]$ due to the continuity in $t$. Then, by considering sequence $n_l$ instead of $n$ and letting $l\to\infty$ for \eqref{remainder part integral form}, we have
\begin{equation*}
\begin{aligned}
\frac{1}{p} \left|w_m^{(\ep)}(t,x)\right|^p \leq I_t^{\alpha}\left[  f_m^{(\ep)}(\cdot,x)  w^{(\ep)}_m(\cdot,x) \left|  w^{(\ep)}_m(\cdot,x)  \right|^{p-2} \right].
\end{aligned}
\end{equation*}
Since $t \leq \tau_m^m$, $\rho_m(u_m) = 1$. By integrating with respect to $x$, we have
\begin{equation}
\label{controlling coefficients - start}
\begin{aligned}
& \frac{\Gamma(\alpha)}{p}\int_{\bR^d}\left| w_m^{(\ep)}(t,x) \right|^p dx \\
&\leq \int_0^t (t-s)^{\alpha-1} \int_{\bR^d} ( Lw_m(s,\cdot) )^{(\ep)}(x) w_m^{(\ep)}(s,x) \left|w_m^{(\ep)}(s,x) \right|^{p-2} dxds \\
&\quad + \int_0^t (t-s)^{\alpha-1} \int_{\bR^d}\left[ \bar b^i(s,\bar x^i)\left(|w_m(s,\cdot)+v(s,\cdot)|^2\right)_{x^i}^{{(\ep)}}(x) \right]  w_m^{(\ep)}(s,x) \left| w_m^{(\ep)}(s,x) \right|^{p-2} dxds.
\end{aligned}
\end{equation}
Furthermore, by integration by parts, we obtain
\begin{equation}
\label{controlling coefficients - 0}
\begin{aligned}
& \int_{\bR^d}\left[ ( Lw_m(s,\cdot) )^{(\ep)}(x) + \bar b^i\left(|w_m(s,\cdot)+v(s,\cdot)|^2\right)_{x^i}^{{(\ep)}}(x) \right]  w_m^{(\ep)}(s,x) \left| w_m^{(\ep)}(s,x) \right|^{p-2} dx \\
& \leq -(p-1)\int_{\bR^d} \left(a^{ij}w_{m}\right)_{x^j}^{(\ep)}(s,x)   \left|w_m^{(\ep)}(s,x) \right|^{p-2} w^{(\ep)}_{mx^i}(s,x) dx \\
&\quad + (p-1)\int_{\bR^d}   \left(\left( 2a^{ij}_{x^j} - b^i \right) w_m\right)^{(\ep)}(s,x)  \left|w_m^{(\ep)}(s,x) \right|^{p-2} w^{(\ep)}_{mx^i}(s,x) dx \\
&\quad + \int_{\bR^d} \left(\left(a^{ij}_{x^ix^j} - b^i_{x^i} + c \right)w_m\right)^{(\ep)}(s,x)  w_m^{(\ep)}(s,x)  \left| w_m^{(\ep)}(s,x) \right|^{p-2}dx\\
&\quad - (p-1)\int_{\bR^d} \bar{b}^i(s,\bar x^i)\left( (w_m(s,\cdot) + v(s,\cdot))^{2} \right)^{(\ep)}(x) \left| w_m^{(\ep)}(s,x) \right|^{p-2} w_{mx^i}^{(\ep)}(s,x)dx.
\end{aligned}
\end{equation}
Additionally, observe that
\begin{equation}
\label{controlling coefficients - 1}
\begin{aligned}
&\left| \left(a^{ij}w_{m}\right)_{x^j}^{(\ep)}(s,x) - a^{ij}(s,x)w_{mx^j}^{(\ep)}(s,x) \right| \\
&\quad = \ep^{-1}\left|\int_{\bR^d}  \left( a^{ij}(s,x-\ep y) - a^{ij}(s,x) \right)w_m(s,x-\ep y)\zeta_{y^j}(y)  dy \right|\\
&\quad \leq N(K)\int_{\bR^d} |w_m(s,x-\ep y)||y||\zeta_{y}(y)|dy,
\end{aligned}
\end{equation}
and by \eqref{ellipticity_of_leading_coefficients_white_noise_in_time},
\begin{equation}
\label{controlling coefficients - 2}
\begin{aligned}
&-\int_{\bR^d}  a^{ij}(s,x) w_{mx^i}^{(\ep)}(s,x)w_{mx^j}^{(\ep)}(s,x) \left|w_m^{(\ep)}(s,x)\right|^{p-2} dx \\
&\quad\leq -K^{-1} \int_{\bR^d} \left|w_m^{(\ep)}(s,x)\right|^{p-2} \left|w_{mx}^{(\ep)}(s,x)\right|^2  dx.
\end{aligned}
\end{equation}
Thus, by combining \eqref{controlling coefficients - 1} and \eqref{controlling coefficients - 2}
\begin{equation}
\label{controlling coefficients - 3}
\begin{aligned}
&-\int_{\bR^d} \left(a^{ij}w_{m}\right)_{x^j}^{(\ep)}(s,x)\left|w_m^{(\ep)}(s,x)\right|^{p-2}  w_{mx^i}^{(\ep)}(s,x)  dx \\
&\quad = -\int_{\bR^d} \left|w_m^{(\ep)}(s,x)\right|^{p-2}  w_{mx^i}^{(\ep)}(s,x) \left[ \left(aw_{m}\right)_{x^j}^{(\ep)}(s,x) - a(s,x)w_{mx^j}^{(\ep)}(s,x) \right] dx \\
&\quad\quad - \int_{\bR^d} a^{ij}(s,x)  w_{mx^i}^{(\ep)}(s,x)w_{mx^j}^{(\ep)}(s,x) \left|w_m^{(\ep)}(s,x)\right|^{p-2} dx \\
&\quad \leq N\int_{\bR^d} \left|w_m^{(\ep)}(s,x)\right|^{p-2} \left| w_{mx^i}^{(\ep)}(s,x) \right|  \int_{\bR^d} |w_m(s,x-\ep y)||y||\zeta_{y}(y)|dy dx \\
&\quad\quad- K^{-1}\int_{\bR^d} \left|w_m^{(\ep)}(s,x)\right|^{p-2}  \left| w_{mx}^{(\ep)}(s,x) \right|^2 dx\\
&\quad \leq N\int_{\bR^d} \left|w_m^{(\ep)}(s,x)\right|^{p-2}  \left( \int_{\bR^d} |w_m(s,x-\ep y)||y||\zeta_{y}(y)|dy\right)^2 dx \\
&\quad\quad- \frac{1}{2}K^{-1}\int_{\bR^d} \left|w_m^{(\ep)}(s,x)\right|^{p-2}  \left| w_{mx}^{(\ep)}(s,x) \right|^2 dx,
\end{aligned}
\end{equation}
where $N = N(K)$.
Moreover,
\begin{equation}
\label{controlling coefficients - 4}
\begin{aligned}
\left| \left(\left(2a^{ij}_{x^j} - b^i\right)w_m\right)^{(\ep)}(s,x) \right| 
&= \left| \int_{\bR^d}\left(2a^{ij}_{y^j}(s,y) - b^i(s,y)\right)w_m(s,y)\zeta_\ep(x-y)dy \right| \\
&\leq K\int_{\bR^d}|w_m(s,y)|\zeta_\ep(x-y)dy \\
&= K(|w_m(s,\cdot)|)^{(\ep)}(x)
\end{aligned}
\end{equation}
and 
\begin{equation}
\label{controlling coefficients - 5}
\left| \left(\left( a^{ij}_{x^ix^j} - b^i_{x^i} + c \right)w_m\right)^{(\ep)}(s,x) \right| \leq  K(|w_m(s,\cdot)|)^{(\ep)}(x).
\end{equation}
Thus, by applying H\"older's inequality, \eqref{controlling coefficients - 3}, \eqref{controlling coefficients - 4}, and \eqref{controlling coefficients - 5} to \eqref{controlling coefficients - 0}, we have
\begin{equation}
\label{controlling coefficients - 6}
\begin{aligned}
&\int_{\bR^d}\left[  ( L w_m(s,\cdot) )^{(\ep)}(x) + \bar b^i\left(|w_m(s,\cdot)+v(s,\cdot)|^2\right)_{x^i}^{{(\ep)}}(x) \right] w^{(\ep)}_m(s,x) \left| w^{(\ep)}_m(s,x) \right|^{p-2} dx \\
&\quad \leq N\int_\bR \left|w_m^{(\ep)}(s,x)\right|^{p-2}  \left( \int_{\bR} |w_m(s,x-\ep y)||y||\zeta_{y}(y)|dy\right)^2 dx \\
&\quad\quad - \frac{p-1}{4K}\sum_{i}\int_\bR \left|w_m^{(\ep)}(s,x)\right|^{p-2}  \left| w_{mx^i}^{(\ep)}(s,x) \right|^2 dx\\
&\quad\quad + N\sum_{i}\int_{\bR^d}  \left( (|w_m(s,\cdot)|)^{(\ep)}(x) \right)^2 \left| w_m^{(\ep)}(s,x) \right|^{p-2}  dx \\
&\quad\quad + N\int_{\bR^d} (|w_m(s,\cdot)|)^{(\ep)}(x)   \left| w_m^{(\ep)}(s,x) \right|^{p-1}dx.\\
&\quad\quad - (p-1)\sum_{i}\int_{\bR^d}  \bar b^i(s,\bar x^i)\left(   (w_m(s,\cdot) + v(s,\cdot))^{2} \right)^{(\ep)}(x)  \left| w_m^{(\ep)}(s,x) \right|^{p-2} w_{mx^i}^{(\ep)}(s,x)dx,
\end{aligned}
\end{equation}
where $N = N(K)$.
Furthermore, note that 
\begin{equation}
\label{int by part}
\int_{\bR^d}\left|w_m^{(\ep)}(s,x)\right|^{p} w_{mx^i}^{(\ep)}(s,x)dx  = 0\quad\text{for}\quad s\leq t.
\end{equation} 
Indeed, take a nonnegative smooth function $\phi\in C_c^\infty(\bR^d)$ such that $\phi(x) = 1$ on $|x|<1$, $\phi(x) = 0$ on $|x|>2$, and $\sup_{x\in\bR^d}|\phi'(x)|\leq 2$. Then, integration by parts yields
\begin{equation*}
\begin{aligned}
&\int_{\bR^d} \left|w_m^{(\ep)}(s,x)\right|^{p} w_{mx^i}^{(\ep)}(s,x)\phi(x/n) dx \\
&\quad = -p\int_{\bR^d} \left|w_m^{(\ep)}(s,x)\right|^{p}w_{mx^i}^{(\ep)}(s,x)\phi(x/n) dx - \frac{1}{n}\int_{\bR^d}\left| w_m^{(\ep)}(s,x) \right|^p w_m^{(\ep)}(s,x)\phi'(x/n)dx.
\end{aligned}
\end{equation*}
Thus, we have
\begin{equation}
\label{proof by int by part}
\limsup_{n\to\infty}\left| \int_{\bR^d} \left|w_m^{(\ep)}(s,x)\right|^{p} w_{mx^i}^{(\ep)}(s,x)\phi(x/n) dx \right|
 \leq \limsup_{n\to\infty}\frac{2}{n(p+1)}\int_{\bR^d}\left| w_m^{(\ep)}(s,x) \right|^{p+1} dx=0
\end{equation}
and \eqref{proof by int by part} yields \eqref{int by part}. Then, from the last term of \eqref{controlling coefficients - 6}, by applying \eqref{int by part} and the H\"older's inequality, we have
\begin{equation}
\label{controlling coefficients - 7}
\begin{aligned}
&\left|\sum_{i} \int_{\bR^{d-1}}\bar b^i(s,\bar x^i)\int_{\bR} \left(\left|  w_m(t,\cdot)+v(t,\cdot)\right|^{2}  \right)^{(\ep)} (x)  \left| w_m^{(\ep)}(t,x) \right|^{p-2} w_{mx^i}^{(\ep)}(t,x) dx^i d\bar x^i \right|\\
&\quad\leq N \sum_{i} \int_{\bR^{d}} \left|\left(\left|  w_m(t,\cdot)+v(t,\cdot)\right|^{2}  \right)^{(\ep)} (x)- \left|w_{m}^{(\ep)}(t,x)\right|^{2}\right|  \left| w_m^{(\ep)}(t,x) \right|^{p-2} \left| w_{mx^i}^{(\ep)}(t,x) \right| dx \\
&\quad \leq N\int_{\bR^d} \left(\left(\left|  w_m(t,\cdot)+v(t,\cdot)\right|^{2}  \right)^{(\ep)} (x)- \left|w_{m}^{(\ep)}(t,x)\right|^{2}\right)^2  \left| w_m^{(\ep)}(t,x) \right|^{p-2} dx \\
&\quad\quad + \frac{1}{8KN} \sum_{i} \int_{\bR^d}\left| w_{mx^i}^{(\ep)}(t,x) \right|^2  \left| w_m^{(\ep)}(t,x) \right|^{p-2} dx,
\end{aligned}
\end{equation}
where $N = N(K)$.
Then, by applying \eqref{controlling coefficients - 6} and \eqref{controlling coefficients - 7} to \eqref{controlling coefficients - start}, we have
\begin{equation*}
\begin{aligned}
&\int_{\bR^d}\left|w_m^{(\ep)}(t,x)\right|^p dx  \\
&\quad\leq NI^\alpha_t \int_\bR \left|w_m^{(\ep)}(s,x)\right|^{p-2}  \left( \int_{\bR} |w_m(s,x-\ep y)||y||\zeta_{y}(y)|dy\right)^2 dx \\
&\quad\quad + NI^\alpha_t\int_{\bR^d}  \left| (|w_m(s,\cdot)|)^{(\ep)}(x) \right|^2 \left| w_m^{(\ep)}(s,x) \right|^{p-2} + (|w_m(s,\cdot)|)^{(\ep)}(x)   \left| w_m^{(\ep)}(s,x) \right|^{p-1} dx \\
&\quad\quad + NI^\alpha_t\int_{\bR^d} \left(\left(\left|  w_m(t,\cdot)+v(t,\cdot)\right|^{2}  \right)^{(\ep)} (x)- \left|w_{m}^{(\ep)}(t,x)\right|^{2}\right)^2  \left| w_m^{(\ep)}(t,x) \right|^{p-2} dx,
\end{aligned}
\end{equation*}
where $N = N(p,K)$.
By letting $\ep\downarrow0$, we have
\begin{equation*}
\begin{aligned}
&\|w_m(t,\cdot)\|_{L_p}^p \\
&\leq NI^\alpha_t \int_\bR \left|w_m(\cdot,x)\right|^{p} dx + NI^\alpha_t\int_{\bR^d} \left( \left|  w_m(\cdot,x)+v(\cdot,x)\right|^{2}  - \left|w_{m}(\cdot,x)\right|^{2}\right)^2  \left| w_m(\cdot,x) \right|^{p-2} dx \\
& \leq NI^\alpha_t \int_\bR \left|w_m(\cdot,x)\right|^{p} dx + NI^\alpha_t\int_{\bR^d} \left| v(\cdot,x) \right|^2\left| w_m(\cdot,x) \right|^p +\left| v(\cdot,x) \right|^{4} \left|w_m(\cdot,x) \right|^{p-2} dx \\
&\leq N\left(1 + \sup_{s\leq t,x\in\bR^d}|v(s,x)|^2 \right) I_t^\alpha \left\| w_m(\cdot,\cdot) \right\|_{L_p}^p + N \sup_{s\leq t}\| v(s,\cdot) \|_{L_{2p}}^{2p}.
\end{aligned}
\end{equation*}
for all $t\leq \tau_m^m$. Then, by Theorem \ref{time fractional integral gronwall ineq}, we obtain
\begin{equation*}
\| w_m(t,\cdot) \|_{L_p}^p \leq N \sup_{s\leq t}\| v(s,\cdot) \|^{2p}_{L_{2p}} \left[ 1+ \sum_{k=0}^\infty\frac{\left(  1+\sup_{s\leq t,x\in\bR^d}|v(s,x)|^2 \right)^k }{\Gamma(k\alpha)} \frac{(\Gamma(\alpha)T^{\alpha})^k}{k\alpha}\right]
\end{equation*}
for all $t\leq \tau_m^m$.  The lemma is proved.

\end{proof}

Finally, we demonstrate that the global solution candidate does not explode in a finite time.
\begin{lemma}
\label{lemma:key_lemma}
For any $T<\infty$, we have
\begin{equation*}
\lim_{R\to\infty}P\left( \left\{ \omega\in\Omega:\sup_{t\leq T,x\in\bR^d}|u(t,x)| > R \right\} \right) = 0.
\end{equation*}

\end{lemma}
\begin{proof}
Let $v$ be the function introduced in Lemma \ref{lemma:noise dominant part}. Define 
\begin{equation*}
\begin{aligned}
\tau^1(S) &:=\inf\left\{ t\geq 0:\| v(t,\cdot) \|_{L_{p}}\geq S \right\}\wedge T,\quad
\tau^2(S) :=\inf\left\{ t\geq 0:\sup_{x\in\bR^d}|v(t,x)|\geq S \right\}\wedge T,
\end{aligned}
\end{equation*}
and
\begin{equation*}
\tau_m^0(S) := \tau_m^m \wedge \tau^1(S) \wedge \tau^2(S),
\end{equation*}
where $\tau_m^m$ is the stopping time introduced in \eqref{stopping_time_taumm}.
Set $r := \frac{p}{p-1}$. 
Then, by Lemmas \ref{lemma:quasi-linear results} and \ref{lemma:properties of bessel potential spaces}, \eqref{pointwise_multiplier}, H\"older inequality, and Minkowski inequality, we have
\begin{equation}
\label{solution norm control}
\begin{aligned}
&\| u_m \|^p_{\cH_{p}^{\gamma}(\tau_m^0(S) )} - N\| u_0 \|^p_{U_p^{\gamma}}\\
&\quad\leq N\left\| \sum_{i}\bar b^i (u_m^2\rho_m(u_m))_{x^i}  \right\|_{\bH_p^{\gamma-2}(\tau_m^0(S))}^p + N\| \sigma(u_m)\bm{\eta} \|_{\bH_p^{\gamma-2+c_0}(\tau_m^0(S),l_2)}^p \\
&\quad\leq N\left\|  u_m^{2} \right\|_{\bH_p^{\gamma-1}(\tau_m^0(S))}^p + N\| \sigma(u_m)\bm{\eta} \|_{\bH_p^{\gamma-2+c_0}(\tau_m^0(S),l_2)}^p \\
&\quad\leq N\bE\int_0^{\tau_m^0(S) }\int_{\bR^d} \left| \int_{\bR^d} R_{1-\gamma}(y)|u_m(s,x-y)|^{2} dy \right|^p dxds \\
&\quad\quad + N\bE\int_0^{\tau_m^0(S) }\int_{\bR^d} \left| \int_{\bR^d} \left| R_{-\gamma+2-c_0}(y) \right|^2|u_m(s,x-y)|^{2} dy \right|^{p/2} dxds \\
&\quad\leq N\bE\int_0^{\tau_m^0(S) }  \int_{\bR^d} \left| \int_{\bR^d} \left| R_{1-\gamma}(y) \right|^r|u_m(s,x-y)|^{r} dy \right|^{p/r} dx \int_{\bR^d} |u_m(s,x)|^{p}  dx  ds \\
&\quad\quad +  N\bE\int_0^{\tau_m^0(S) }  \left( \int_{\bR^d} \left|R_{-\gamma+2-c_0}(x)\right|^2 dx\right)^{p/2} \int_{\bR^d} |u_m(s,x)|^{p}  dx  ds \\
&\quad\leq N_0\bE\int_0^{\tau_m^0(S)} \left[ 1 +  \int_{\bR^d} |u_m(s,x)|^{p} dx \right]  \int_{\bR^d} |u_m(s,x)|^p dx ds,
\end{aligned}
\end{equation}
where $N_0 = N(\alpha,\beta,\gamma,d,p,K,T)\left[ \left(\int_{\bR^d} \left|R_{1-\gamma}(x)\right|^r dx\right)^{p/r}+\left( \int_{\bR^d} \left|R_{-\gamma+2-c_0}(x)\right|^2 dx\right)^{p/2} \right]$. Note that $N_0<\infty$ due to $r < \frac{d}{d+\gamma-1} $ and Remark \ref{Kernel}. Then, by Lemma \ref{lemma:nonlinear control} and the definitions of $\tau_1(S)$ and $\tau_2(S)$,
\begin{equation}
\label{lp norm is bounded}
\begin{aligned}
&\int_{\bR^d} |u_m(t,x)|^p dx \\
&\leq N\int_{\bR^d} |u_m(t,x)-v(t,x)|^p + |v(t,x)|^p dx \\
&\leq N \sup_{s\leq t}\sup_{x\in\bR^d}|v(t,x)|\sup_{s\leq t}\| v(s,\cdot) \|_{p}^{p} \left[ 1+ \sum_{k=0}^\infty\frac{\left(  1+\sup_{s\leq t,x\in\bR^d}|v(t,x)|^2 \right)^k }{\Gamma(k\alpha)} \frac{(\Gamma(\alpha)T^{\alpha})^k}{k\alpha}\right] \\
&\quad  + \int_{\bR^d} |v(t,x)|^p dx \\
&<N(p,S,K).
\end{aligned}
\end{equation}
Therefore, by combining \eqref{solution norm control} and \eqref{lp norm is bounded}, we have
\begin{equation}
\label{solution norm of u is bounded}
\| u_m \|^p_{\cH_{p}^{\gamma}(\tau_m^0(S) )}  \leq N + N\| u_0 \|^p_{U_p^{\gamma}},
\end{equation}
where $N = N(\alpha,\beta,\gamma,d,p,S,K,T)$.
It should be noted that the right-hand side of \eqref{solution norm of u is bounded} is independent of $m$. Therefore, by Chebyshev's inequality and Lemma \ref{lemma:existence uniqueness reg of local solution}, we have
\begin{equation*}
\begin{aligned}
P\left( \sup_{t\leq \tau^0_m(S)}\sup_{x\in\bR^d}|u(t,x)| > R \right) 
&\leq \frac{1}{R^p}\bE\sup_{t\leq \tau^0_m,x\in\bR^d}|u(t,x)|^p  \\
&\leq \frac{1}{R^p}\bE\sup_{t\leq \tau^0_m,x\in\bR^d}|u_m(t,x)|^p  \\
&\leq \frac{1}{R^p}\| u_m \|_{\cH_p^{\gamma}(\tau^0_m)}^p  \\
&\leq \frac{N}{R^p},
\end{aligned}
\end{equation*}
where $N = N(u_0,\alpha,\beta,\gamma,d,p,S,K,T).$
In contrast, by Lemma \ref{lemma:noise dominant part},
\begin{equation*}
\begin{aligned}
& P\left(\tau^1(S)<T\right) + P\left(\tau^2(S)<T\right) \\
&\leq P\left( \sup_{t\leq T}\| v(t,\cdot)\|_{L_{p}} > S \right) + P\left( \sup_{t\leq T}\sup_{x\in\bR^d} |v(t,x)| > S \right) \\
&\leq \frac{1}{S^p}\bE\sup_{t\leq T}\| v(t,\cdot) \|_{L_{p}}^{p} + \frac{1}{S^p}\bE\sup_{t\leq T,x\in\bR^d} |v(t,x)|^p \\
&\leq \frac{1}{S^p}N(u_0,h,\alpha,\beta,\gamma,d,p,K,T).
\end{aligned}
\end{equation*}
Thus, 
\begin{equation*}
\begin{aligned}
&P\left(\sup_{t\leq T,x\in\bR^d}|u(t,x)| > R\right)  \\
&\leq \liminf_{m\to\infty}P\left(\sup_{t\leq \tau_m^0(S),x\in\bR^d}|u(t,x)| > R\right) + P\left(\tau^1(S)<T\right) + P\left(\tau^2(S)<T\right) \\
&\leq \frac{N_1}{R^p} + \frac{N_2}{S^p},
\end{aligned}
\end{equation*}
where $N_1 = N_1(u_0,\alpha,\beta,\gamma,d,p,S,K,T)$ and $N_2 = N_2(u_0,h,\alpha,\beta,\gamma,d,p,K,T)$.
The lemma is proved by letting $R\to\infty$ and $S\to\infty$ in order.

\end{proof}

\begin{proof}[\bf{Proof of Theorem \ref{main theorem}}]

{\it Step 1. (Uniqueness). }
Suppose $u,\bar u\in \cH_{p,loc}^{\gamma}$ are nonnegative solutions of equation \eqref{eq:target equation}. By Definition \ref{def:solution space}, there are bounded stopping times $\tau_n$ $(n = 1,2,\cdots)$ 
such that
\begin{equation*}
\tau_n\uparrow\infty\quad\mbox{and}\quad u, \bar u \in \cH_{p}^{\gamma}(\tau_n).
\end{equation*}
Fix $n\in\bN$. Note that $u,\bar{u} \in C([0,\tau_n];C(\bR^d))$ almost surely and
\begin{equation}
\label{embedding in the proof of uniqueness_infinite_noise_1}
\bE\sup_{t\leq\tau_n}\sup_{x\in\bR^d}|u(t,x)|^p + \bE\sup_{t\leq\tau_n}\sup_{x\in\bR^d}|\bar u(t,x)|^p < \infty.
\end{equation}
Then, for $m \in \bN$, define
\begin{equation*}
\begin{gathered}
\tau_{m,n}^1:=\inf\left\{t\geq0:\sup_{x\in\bR^d}|u(t,x)|> m\right\}\wedge\tau_n, \\
\tau_{m,n}^2:=\inf\left\{t\geq0:\sup_{x\in\bR^d}|\bar u(t,x)|> m\right\}\wedge\tau_n,
\end{gathered}
\end{equation*}
and 
\begin{equation} 
\label{stopping_time_cutting}
\tau_{m,n}:=\tau_{m,n}^1\wedge\tau_{m,n}^2.
\end{equation}
Due to \eqref{embedding in the proof of uniqueness_infinite_noise_1}, $\tau_{m,n}^1$ and $\tau_{m,n}^2$ are well-defined stopping times; thus, $\tau_{m,n}$ is a stopping time.
Observe that $u,\bar u\in\cH_p^{\gamma}(\tau_{m,n})$ and  $\tau_{m,n}\uparrow \tau_n$ as $m\to\infty$ almost surely. Fix $m\in\bN$. Notice that $u,\bar u\in\cH_p^{\gamma}(\tau_{m,n})$ are solutions to equation
\begin{equation*}
\partial_t^\alpha v = Lv + \bar{b}^i \left( v^{2}\rho_m(v) \right)_{x^i} +\partial_t^\beta\int_0^t \sigma(v)  dW_t, \quad 0<t\leq\tau_{m,n};\quad v(0,\cdot)=u_0,
\end{equation*}
where $Lv = a^{ij}v_{x^ix^j} + b^i v_{x^i} + cv$.
By the uniqueness result in Lemma \ref{lemma:existence uniqueness reg of local solution}, we conclude that $u=\bar u$ in $\cH_p^{\gamma}(\tau_{m,n})$ for each $m\in\bN$. The monotone convergence theorem yields $u=\bar u$ in $\cH_p^{\gamma}(\tau_n)$, which implies $u = \bar u$ in $\cH_{p,loc}^{\gamma}$.

\vspace{2mm}

{\it Step 2 (Existence.)}.  
Let $T<\infty$. For $m\in\bN$, define $\tau_m^m$ and $u$ as in Remark \ref{def_u}. Observe that
\begin{equation*}
\begin{aligned}
&P\left(  \tau_m^m < T \right) \leq P\left(   \sup_{t\leq T,x\in\bR^d}|u(t,x)| \geq m \right).
\end{aligned}
\end{equation*}
Indeed, if $\tau_m^m < T$, then $\sup_{t\leq\tau_m^m,x\in\bR^d}|u(t,x)| = \sup_{t\leq\tau_m^m,x\in\bR^d}|u_m(t,x)| = m$ almost surely. Then, by Lemma \ref{lemma:key_lemma}, we have
\begin{equation*}
\begin{aligned}
\limsup_{m\to\infty} P\left(\tau_m^m < T\right) 
&\leq \limsup_{m\to\infty} P\left(\sup_{t\leq T,x\in\bR^d}|u(t,x)|\geq m\right) 
&= 0
\end{aligned}
\end{equation*}
Since $T<\infty$ is arbitrary, $\tau_m^m\to\infty$ in probability. In addition, we conclude that $\tau_m^m \uparrow \infty$ almost surely, because $\tau_m^m$ is increasing in $m$.

Last, set $\tau_m := \tau_m^m\wedge m$. Note that (see Remark \ref{def_u})
$$ u(t,x) = u_m(t,x)\quad\text{for}\quad t\in[0,\tau_m].
$$
Observe that $\sup_{x\in\bR^d}|u_m(t,x)|\leq m$ for $t\in[0,\tau_m]$; thus, $u_m$ satisfies \eqref{eq:target equation} almost everywhere $t\in[0,\tau_m]$ almost surely. Because $u = u_m$ for $t\in[0,\tau_m]$ and $u_m\in\cH_p^\gamma(\tau_m)$, it follows that $u\in\cH_p^{\gamma}(\tau_m)$ and $u$ satisfies \eqref{eq:target equation} for all $t\leq\tau_m$ almost surely. We have $u \in \cH_{p,loc}^\gamma$ because $\tau_m\uparrow\infty$ as $m\to\infty$ almost surely. 

To obtain H\"older regularity, we employ Corollary \ref{corollary:continuous embedding of solution}. The theorem is proved.

\end{proof}

\begin{proof}[\bf{Proof of Theorem \ref{uniqueness_in_p}}]
The proof of Theorem \ref{uniqueness_in_p} is motivated by \cite[Corollarly 5.11]{kry1999analytic}. Since $q>p$, by Theorem \ref{main theorem}, there exists a unique solution $\bar{u}\in \cH_{q,loc}^{\gamma}$ satisfying equation \eqref{eq:target equation}. By Definition \ref{def:solution space}, there exists $\tau_n$ such that $\tau_n\to\infty$ almost surely as $n\to\infty$, $u\in\cH_p^{\gamma}(\tau_n)$ and $\bar{u} \in\cH_q^{\gamma}(\tau_n)$. Fix $n\in\bN$. Because $\frac{2+\alpha d}{\alpha \gamma} < p < q$, we can define $\tau_{m,n}$ $(m\in\bN)$ as in \eqref{stopping_time_cutting}. 
Notice that for any $p_0 > p$, we have
\begin{equation*}
u\in \bL_{p_0}(\tau_{m,n}) 
\end{equation*}
since
\begin{equation*}
\bE\int_0^{\tau_{m,n}} \int_{\bR} |u(t,x)|^{p_0} dxdt \leq m^{p_0-p}\bE\int_0^{\tau_{m,n}} \int_{\bR} |u(t,x)|^{p} dxdt < \infty.
\end{equation*} 
Observe that $\bar{b}^i u u_{x^i} \in \bH_q^{\gamma-2}(\tau_{m,n}) \subset \bH_q^{-2}(\tau_{m,n})$. Indeed, similar to \eqref{estimate for local sol 1},
\begin{equation*}
\begin{aligned}
\bE \int_0^{\tau_{m,n}} \left\| \frac{1}{2}\bar{b}^i(s,\cdot)  \left( (u(s,\cdot))^2 \right)_{x^i} \right\|_{H_q^{\gamma-2}}^q ds  
&\leq N \bE \int_0^{\tau_{m,n}} \int_{\bR}|u(s,x)|^{2q} dxds 
< \infty.
\end{aligned}
\end{equation*}
Additionally, we have
\begin{equation*}
\begin{aligned}
&au_{xx} \in \bH_{q}^{-2}(\tau_{m,n}), \quad b u_{x} \in \bH_{q}^{-1}(\tau_{m,n}), \quad\text{and}\quad c u \in \bL_{q}(\tau_{m,n}).
\end{aligned}
\end{equation*}
Therefore, because $\bL_{q}(\tau_{m,n})\subset \bH_{q}^{-1}(\tau_{m,n}) \subset \bH_{q}^{-2}(\tau_{m,n})$,
\begin{equation} \label{deterministic_part_of_u_infinite_noise_1}
a^{ij}u_{x^ix^j} + b^i u_{x^i} + c u + \bar b^i u u_{x^i} \in \bH_{q}^{-2}(\tau_{m,n}).
\end{equation}
Similar to \eqref{bound of stochastic part}, we have
\begin{equation} 
\label{proof_of_claim_stochastic_part}
\| \sigma(u) \bm{\eta} \|^q_{\bH^{\gamma-2+c_0}_q(\tau_{m,n},l_2)}  \leq N\int_0^{\tau_{m,n}}\left\| h(t,\cdot) \right\|_{L_q}^q dt < \infty.
\end{equation}
Thus, we have
\begin{equation} \label{stochastic_part_of_u_infinite_noise_1}
\sigma( u )\bm{\eta} \in \bH_q^{\gamma-2+c_0}(\tau_{m,n},l_2) \subset \bH_{q}^{-2+c_0}(\tau_{m,n},l_2).
\end{equation}
Due to \eqref{deterministic_part_of_u_infinite_noise_1}, \eqref{stochastic_part_of_u_infinite_noise_1}, and Lemma \ref{lemma:quasi-linear results}, $u$ is in $\cL_q(\tau_{m,n})$ and $u$ satisfies \eqref{eq:target equation} for almost everywhere $t\leq \tau_{m,n}$ almost surely. On the other hand, since $\bar{b}^i uu_{x^i} \in \bH_q^{\gamma-2}(\tau_{m,n})$ and $\sigma( u )\bm{\eta} \in \bH^{\gamma-2+c_0}_q(\tau_{m,n},l_2) $, Lemma \ref{lemma:quasi-linear results} implies that there exists $v\in \cH_q^{\gamma}(\tau_{m,n})$ satisfying
\begin{equation}
\label{equation_in_proof_of_consistency}
\partial_t^\alpha v = Lv +  \bar{b}^i uu_{x^i} + \partial_t^\alpha \int_0^t \sigma(u) \eta_k dw^k_t, \quad 0<t\leq\tau_{m,n}\,; \quad v(0,\cdot)=u_0,
\end{equation}
where $Lv = a^{ij}v_{x^ix^j} + b^i v_{x^i} + cv $.
In \eqref{equation_in_proof_of_consistency}, note that $\bar{b}^i uu_{x^i}$ and $\sigma^k(u)$ are used instead of $\bar{b}^i vv_{x^i}$ and $\sigma^k(v)$. Moreover, because $u\in\cL_q(\tau_{m,n})$ satisfies equation \eqref{equation_in_proof_of_consistency}, $\bar v := u-v\in \cL_q(\tau_{m,n})$ satisfies 
\begin{equation*}
\partial_t^\alpha \bar v = a^{ij}\bar v_{x^ix^j} + b^i \bar v_{x^i} + c\bar v, \quad 0<t\leq\tau_{m,n}\,; \quad \bar v(0,\cdot)=0. 
\end{equation*}
By the deterministic version of Lemma \ref{lemma:quasi-linear results}, we have $\bar v = 0$ in $\cL_q(\tau_{m,n})$; thus, $u = v$ in $\cL_q(\tau_{m,n})$. Therefore, $u$ is in $\cH_q^{\gamma}(\tau_{m,n})$. As $\bar{u}\in \cH_q^{\gamma}(\tau_{m,n})$ and $\bar{u}$ satisfies equation \eqref{eq:target equation}, by Lemma \ref{lemma:existence uniqueness reg of local solution}, we have $u = \bar{u}$ in $\cH_q^{\gamma}(\tau_{m,n})$.  The theorem is proved.
\end{proof}

\vspace{2mm}


\section{ Proof of Theorem \texorpdfstring{\ref{thm:embedding theorem for sol space} }{Lg} }
\label{section: proof of continuity of the solution in t}

This section provides a proof of the embedding theorem for solution spaces $\cH_p^\gamma(\tau)$. Consider the following fractional diffusion equation
\begin{equation} 
\label{eq:example of fractional diffusion-wave equation}
\partial_t^\alpha u =  \Delta u \quad t>0\,
;\quad u(0,\cdot) = u_0(\cdot),
\end{equation}
where $\alpha\in(0,1)$ and $u_0(\cdot)\in C_c^\infty(\bR^d)$.
It turns out that a fundamental solution $p(t,x)$ exists such that 
\begin{equation*}
p(t,\cdot)\in L_1(\bR^d)\quad \mbox{and}\quad \cF(p(t,\cdot))(\xi) = E_\alpha(-t^\alpha|\xi|^2)
\end{equation*}
(e.g., \cite[Theorem 2.1]{kim2015asymptotic}), and the solution of \eqref{eq:example of fractional diffusion-wave equation} is given by
\begin{equation*}
\begin{aligned}
u(t,x) &= (u_0\ast p(t,\cdot))(x) = \int_{\bR^d}u_0(y)p(t,x-y)dy.
\end{aligned}
\end{equation*}
For convenience, define
\begin{equation*}
q_{\alpha,\beta}(t,x):=
\begin{cases}
I^{\alpha - \beta}p(t,x)\quad \quad &\text{if}\quad \quad\alpha\geq\beta \\
D^{\beta - \alpha}p(t,x)\quad \quad &\text{if}\quad \quad\alpha < \beta.
\end{cases}
\end{equation*}
We gather some facts related to $p(t,x)$ and $q_{\alpha,\beta}(t,x)$ (for more information, see \cite{kim2017lq,kim2019sobolev,kim2015asymptotic}).

\begin{lemma}
\label{lemma:properties of kernel}
Let $d\in\bN$, $\alpha\in(0,1)$, $\beta < \alpha + 1/2$, $\gamma\in[0,2)$, and $\sigma\in\bR$. 
\begin{enumerate}[(i)]

\item 
\label{item: prop of fundamental solution}
For all $t\neq 0$ and $x\neq 0$,
\begin{equation*}
\partial_t^\alpha p(t,x) = \Delta p(t,x)\quad \text{and} \quad \partial_t p(t,x) = \Delta q_{\alpha,1}(t,x).
\end{equation*}
Additionally, for each $x\neq 0$, $\frac{\partial}{\partial t}p(t,x)\to0$ as $t \downarrow 0$. Moreover, $\frac{\partial}{\partial t}p(t,x)$ is integrable in $\bR^d$ uniformly on $t\in [\delta,T]$ for any $\delta>0$.

\item 
\label{item:upper bound properties of kernel_1}
There exist constants $c = c(\alpha,d)$ and $N = N(\alpha,d)$ such that if $|x|^2 \geq t^\alpha$,
$$ |p(t,x)| \leq N|x|^{-d}\exp\left\{ -c|x|^{\frac{2}{2-\alpha}}t^{-\frac{\alpha}{2-\alpha}} \right\}.
$$

\item 
\label{item:upper bound properties of kernel_2}
Let $n\in \bN$. Then, there exists $N = N(\alpha,\gamma,n)$ such that
\begin{equation*}
\left| D_t^\sigma D_x^n (-\Delta)^{\gamma/2}q_{\alpha,\beta}(1,x)\right| \leq N\left( |x|^{-d+2-\gamma-n}\wedge |x|^{-d-\gamma-n} \right).
\end{equation*}

\item 
\label{item:scaling properties of kernel}
The scaling properties hold. In other words,
\begin{equation*}
\begin{gathered}
q_{\alpha,\beta}(t,x) = t^{-\frac{\alpha d}{2}+\alpha-\beta}q_{\alpha,\beta}(1,xt^{-\frac{\alpha}{2}}),\\
 (-\Delta)^{\gamma/2}q_{\alpha,\beta}(t,x) = t^{-\frac{\alpha(d+\gamma)}{2}+\alpha-\beta} (-\Delta)^{\gamma/2}q_{\alpha,\beta}(1,xt^{-\frac{\alpha}{2}}).
\end{gathered}
\end{equation*}

\end{enumerate}
\end{lemma}
\begin{proof}
To see \eqref{item: prop of fundamental solution}, \eqref{item:upper bound properties of kernel_1}, and \eqref{item:upper bound properties of kernel_2} follow from Theorems 2.1 and 2.3 of \cite{kim2019sobolev}. For \eqref{item:scaling properties of kernel}, see (5.2) in \cite{kim2019sobolev}.
\end{proof}

\begin{remark}
\label{intro operators}
To prove Theorem \ref{thm:embedding theorem for sol space}, we define the operators. Let $\phi\in C_c^\infty(\bR^d)$ and $f\in C_c^\infty((0,\infty)\times\bR^d)$. Take a function $g = (g^1,g^2,\dots)$ satisfying the form
\begin{equation}
\label{form of g}
g^k(t,x) = 
\begin{cases}
\sum_{i = 1}^{n} 1_{\opar\tau_{i-1},\tau_{i}\cbrk}(t)g^{ik}(x)  &\quad\mbox{for}\quad k = 1,2,\dots,n,\\
\quad\quad\quad\quad 0 &\quad\mbox{for}\quad k = n+1,\dots
\end{cases}
\end{equation}
for some $n\in\bN$, where $\tau_i$ is the bounded stopping time and $g^{ik}\in C_c^\infty(\bR^d)$. Further, we set
\begin{gather}
T^1_t \phi(x) := \int_{\bR^d}p(t,x-y)\phi(y)dy, \label{def of t1}\\
T^2_t f(t,x) := \int_0^t \int_{\bR^d} q_{\alpha,1}(t-r,x-y) f(r,y) dy ds, \label{def of t2}\\
T^3_t g(t,x) := \sum_{k}\int_0^t \int_{\bR^d} q_{\alpha,\beta}(t-r,x-y) g^k(r,y) dy dw_s^k. \label{def of t3}
\end{gather}
It is well-known that $T_t^1 \phi$, $T_t^2 f$, and $T_t^3 g$ are solutions to 
\begin{equation}
\label{eq only initial data}
\partial_t^\alpha u^1 = \Delta u^1;\quad u^1(0,\cdot) = \phi,
\end{equation}
\begin{gather*}
\partial_t^\alpha u^2 = \Delta u^2 + f;\quad u^2(0,\cdot) = 0, \label{eq only external force}\\
\partial_t^\alpha u^3 = \Delta u^3 + \partial^\beta_t\int_0^t g^kdw_s^k;\quad u^3(0,\cdot) = 0, \label{eq only random force}\\
\end{gather*}
respectively (for more information, see \cite{kim2017lq,kim2019sobolev,kim2015asymptotic}).

\end{remark}

First, we provide a smoothing effect of $T_t^1$, which implies the continuity of $T_t^1\phi$ in $t$.

\begin{lemma} 
\label{sup norm control lemma}
Let $\tau\leq T$ be a bounded stopping time and $\alpha\in(0,1)$. If $p\in (1,\infty)$, $\theta\in[0,\alpha)$, $\phi\in C_c^\infty(\bR^d)$, and $t\in(0,T)$, we have
\begin{equation} 
\label{ineq:smoothing effect 1}
e^{-t}\left\| T^1_t \phi \right\|_{H^{\frac{2\theta}{\alpha}}_p}\leq Nt^{-\theta}\| \phi \|_{L_p}
\end{equation}
where $N = N(\alpha,\theta,d,p,T)$.

\end{lemma}
\begin{proof}

In the case of $\theta=0$, by Mink\"owski's inequality, we have
\begin{equation} 
\label{ineq:minkowski ineq of kernel}
\begin{aligned}
\| T^1_t \phi \|_{L_p} \leq \| p(t,\cdot)\ast \phi \|_{L_p} \leq \|p(t,\cdot)\|_{L_1}\| \phi \|_{L_p}\leq \| \phi \|_{L_p}.
\end{aligned}
\end{equation}
Thus, we have \eqref{ineq:smoothing effect 1}. For $\theta\in(0,\alpha)$, observe that 
\begin{equation*}
\| e^{-t}T^1_t\phi \|_{H_p^{2\theta/\alpha}}  = \|(1-\Delta)^{\theta/\alpha}(e^{-t}T^1_t\phi)\|_{L_p}  \leq \| e^{-t}T^1_t\phi \|_{L_p}+\| (-\Delta)^{\theta/\alpha}(e^{-t}T^1_t\phi) \|_{L_p},
\end{equation*}
where the last inequality follows from Lemma \ref{lemma:properties of bessel potential spaces} \eqref{multiplier theorem}.
As $e^{-t}\leq (N t^{-\theta}\wedge 1)$, we have
\begin{equation}
\label{proof of smoothing ef - 1}
\left\| e^{-t}T^1_t \phi \right\|_{H_p^{2\theta/\alpha}}   \leq N\left(t^{-\theta}\left\| T^1_t\phi \right\|_{L_p} + \left\| (-\Delta)^{\theta/\alpha}T^1_t\phi \right\|_{L_p}\right).
\end{equation}
By inequality \eqref{ineq:minkowski ineq of kernel}, we have
\begin{equation}
\label{proof of smoothing ef - 2}
t^{-\theta}\| T^1_t\phi \|_{L_p} \leq t^{-\theta}\| \phi \|_{L_p}.
\end{equation}
On the other hand, Minkowski's inequality yields
\begin{equation}
\label{proof of smoothing ef - 3}
\| (-\Delta)^{\theta/\alpha}T^1_th \|_{L_p} = \| (-\Delta)^{\theta/\alpha}(p(t,\cdot)\ast \phi) \|_{L_p} \leq  \| ((-\Delta)^{\theta/\alpha}p)(t,\cdot) \|_{L_1}\| \phi \|_{L_p}.
\end{equation}
Additionally Lemma \ref{lemma:properties of kernel} \eqref{item:scaling properties of kernel}, \eqref{item:upper bound properties of kernel_1}, and \eqref{item:upper bound properties of kernel_2} imply 
\begin{equation}
\label{proof of smoothing ef - 4}
\begin{aligned}
\| ((-\Delta)^{\theta/\alpha}p)(t,\cdot) \|_{L_1} 
&= \int_{\bR^d} |((-\Delta)^{\theta/\alpha}p)(t,x)|dx \\
&\leq t^{-\theta}\int_{\bR^d}|((-\Delta)^{\theta/\alpha}p)(1,x)|dx \\
&\leq N(\alpha,\theta,d,p)t^{-\theta}.
\end{aligned}
\end{equation}
Then, by applying \eqref{proof of smoothing ef - 4} to \eqref{proof of smoothing ef - 3}, we have
\begin{equation}
\label{proof of smoothing ef - 5}
\| (-\Delta)^{\theta/\alpha}T^1_t\phi \|_{L_p} \leq N t^{-\theta} \| \phi \|_{L_p}.
\end{equation}
Thus, by plugging in \eqref{proof of smoothing ef - 2} and \eqref{proof of smoothing ef - 5} into \eqref{proof of smoothing ef - 1}, we have \eqref{ineq:smoothing effect 1}. 
The lemma is proved.
\end{proof}

To proceed further, we introduce one of the embedding theorems for Slobodetskii's spaces.

\begin{lemma} 
\label{slobodetskii's embedding}
If $\mu p >1$ and $p\geq 1$, for any continuous $L_p$-valued function $\phi$ and $\gamma\leq \rho$, we have the following:
\begin{equation}
\label{ineq:slobodetskii's embedding 1}
\begin{aligned}
\| \phi(\rho) - \phi(\gamma) \|_{L_p}^p \leq N(\rho-\gamma)^{\mu p -1}\int_\gamma^\rho\int_\gamma^\rho 1_{t > s}\frac{\| \phi(t) - \phi(s) \|_{L_p}^p}{|t-s|^{1+\mu p}} dsdt\quad\left( \frac{0}{0} := 0 \right),
\end{aligned}
\end{equation}
where $N = N(\mu,p)$. In particular,
\begin{equation}
\label{ineq:slobodetskii's embedding 2}
\begin{aligned}
\bE \sup_{0\leq s < t \leq T}\frac{\| \phi(t) - \phi(s) \|_{L_p}^p}{|t-s|^{\mu p - 1}} 
&\leq N \int_0^{T}\int_0^{T} 1_{t>s} \frac{\bE\left\| \phi(t) - \phi(s) \right\|_{L_p}^p}{|t-s|^{1 + \mu p}}dsdt.
\end{aligned}
\end{equation}

\end{lemma}

With the help of Lemma \ref{slobodetskii's embedding}, we obtain the H\"older continuity of $T_t^1\phi$, $T_t^2f$, and $T_t^3g$ on $t\in[0,T]$.

\begin{lemma} 
\label{holder semi norm control lemma}
Let $\alpha\in(0,1)$.
\begin{enumerate}[(i)]

\item 
\label{item: holder conti of T1}
If $p \in (1,\infty)$, $\theta\in(0,\alpha)$, and $\mu\in(0,1)$ satisfy
\begin{equation}
\label{condition of p and mu}
p > \frac{1}{\alpha-\theta},\quad  \mu\in \left( \frac{1}{\alpha p}, \frac{\alpha-\theta}{\alpha} \right)
\end{equation}
and $\phi\in C_c^\infty(\bR^d)$, then for $t,s\in(0,T)$,
\begin{equation}
\label{eq: holder conti of T1-1}
\left\| T^1_{t} f - T^1_{s} f \right\|_{H^{\frac{2\theta}{\alpha}-2}_p}^p 
\leq N|t-s|^{\alpha\mu p - 1}\left\| \phi \right\|_{L_p}^p,
\end{equation}
where $N = N(\alpha,\theta,d,p,T)$. Additionally,
\begin{equation} 
\label{eq: holder conti of T1-2}
\sup_{0 \leq s < t \leq T}\frac{\left\| T_{t}^1 \phi - T_{s}^1 \phi \right\|_{H_p^{\frac{2\theta}{\alpha}-2}}^p}{|t-s|^{\alpha\mu p - 1}} \leq N\left\| \phi \right\|_{L_p}^p
\end{equation}
where $N = N(\alpha,\theta,d,p,T)$.

\item 
\label{item: holder conti of T2}
If $p\in(1,\infty)$, $\theta\in(0,\alpha)$, and $\mu \in (0,1)$ satisfy \eqref{condition of p and mu} and $f\in C_c^\infty((0,T)\times\bR^d)$, then for $t,s\in(0,T)$,
\begin{equation}
\label{eq: holder conti of T2-1}
\left\| T^2_{t} f - T^2_{s} f \right\|_{H^{2\theta/\alpha}_p}^p 
\leq N|t-s|^{\alpha\mu p - 1}\int_0^T\left\| f(r,\cdot) \right\|_{L_p}^pdr
\end{equation}
where $N = N(\alpha,\theta,d,p,T)$. Additionally, 
\begin{equation}
\label{eq: holder conti of T2-2}
\sup_{0 \leq s < t \leq T}\frac{\left\| T^2_{t} f - T^2_{s} f \right\|_{H^{2\theta/\alpha}_p}^p}{|t-s|^{\alpha\mu p - 1}} 
\leq N\int_0^T\left\| f(r,\cdot) \right\|_{L_p}^pdr,
\end{equation}
where $N = N(\alpha,\theta,d,p,T)$.

\item 
\label{item: holder conti of T3}
Let $\tau\leq T$ be a bounded stopping time. If $\beta < \alpha + \frac{1}{2}$, $p\in[2,\infty)$, $\theta\in(0,\alpha -\beta  + 1/2)$, and $\mu\in(0,1)$ satisfy
$$ p > \frac{1}{(\alpha-\beta-\theta)\wedge 1/2 + 1/2},\quad  \mu\in\left( \frac{1}{\alpha p}, \frac{(\alpha  - \beta - \theta)\wedge 1/2 + 1/2}{\alpha} \right),
$$
and $g = (g^1,g^2,\dots)$ is a function of type \eqref{form of g}, then, for $t,s\in(0,T)$, we have
\begin{equation}
\label{eq: holder conti of T3}
\bE \left\| T_{t\wedge \tau}^3 g - T_{s\wedge \tau}^3 g \right\|_{H^{2\theta/\alpha}_p(l_2)}^p 
 \leq N|t-s|^{\alpha \mu p - 1}\| g \|_{\bL_p(\tau,l_2)}^p,
\end{equation}
where $N = N(\alpha,\beta,\theta,d,p,T)$.
Additionally,
\begin{equation}
\label{eq: holder conti of T3-2}
\bE\sup_{0 \leq s < t \leq \tau}\frac{ \left\| T_{t}^3 g - T_{s}^3 g \right\|_{H^{2\theta/\alpha}_p(l_2)}^p }{|t-s|^{\alpha \mu p - 1}}
\leq N\| g \|_{\bL_p(\tau,l_2)}^p,
\end{equation}
where $N = N(\alpha,\beta,\theta,d,p,T)$.

\end{enumerate}
\end{lemma}
\begin{proof}[\textit{Proof of \eqref{item: holder conti of T1}}] By \eqref{ineq:slobodetskii's embedding 1}, we have
\begin{equation}
\label{proof of holder conti - 1}
\begin{aligned}
\left\| T_\rho^1 \phi - T_\gamma^1 \phi \right\|_{H_p^{\frac{2\theta}{\alpha}-2}}^p 
&\leq N |\rho-\gamma|^{\alpha\mu p - 1 }\int_0^{T}\int_0^{T} 1_{t>s} \frac{\left\| T_{t}^1 \phi - T_{s}^1 \phi \right\|_{H_p^{\frac{2\theta}{\alpha}-2}}^p}{|t-s|^{1 + \alpha\mu p}}dsdt\\
&\leq N |\rho-\gamma|^{\alpha\mu p - 1 } \int_0^{T}\int_0^{T-t} t^{-1-\alpha \mu p} \left\| T_{t+s}^1 \phi - T_{s}^1 \phi \right\|_{H_p^{\frac{2\theta}{\alpha}-2}}^pdsdt.
\end{aligned}
\end{equation}
Since $T_t^1\phi$ is a solution to \eqref{eq only initial data}, Lemma \ref{sup norm control lemma} yields 
\begin{equation}
\label{proof of holder conti - 2}
\begin{aligned}
&\left\| T_{t+s}^1\phi - T_s^1\phi \right\|_{H_p^{\frac{2\theta}{\alpha}-2}} \\
&\quad\leq \left\| (T_{t+s}^1\phi - \phi) - (T_s^1\phi - \phi) \right\|_{H_p^{\frac{2\theta}{\alpha}-2}} \\
&\quad\leq \left\| \int_0^{t+s}(t+s-r)^{\alpha-1}\Delta T_r^1\phi dr - \int_0^s (s-r)^{\alpha-1}\Delta T_r^1\phi dr \right\|_{H_p^{\frac{2\theta}{\alpha}-2}} \\
&\quad\leq \int_s^{t+s}(t+s-r)^{\alpha-1}\left\|  T_r^1\phi \right\|_{H_p^{\frac{2\theta}{\alpha}}} dr  \\
&\quad\quad + \int_0^s \left|  (t+s-r)^{\alpha-1} - (s-r)^{\alpha-1}\right|\left\|T_r^1\phi\right\|_{H_p^{\frac{2\theta}{\alpha}}}  dr  \\
& \quad \leq N(2t^\alpha+s^\alpha-(t+s)^{\alpha})t^{-\theta}\| \phi \|_{L_p}.
\end{aligned}
\end{equation}
Since $0\leq 2t^\alpha+s^\alpha-(t+s)^{\alpha} \leq 2t^\alpha$, by applying \eqref{proof of holder conti - 2} to \eqref{proof of holder conti - 1},
\begin{equation*}
\begin{aligned}
\left\| T_t^1 \phi - T_s^1 \phi \right\|_{H_p^{\frac{2\theta}{\alpha}-2}}^p
&\leq N |\rho-\gamma|^{\alpha \mu p - 1} \int_0^{T} t^{-1-\alpha \mu p + \alpha p - \theta p}dt\|\phi\|_{L_p}^p \\
&\leq N |\rho-\gamma|^{\alpha \mu p - 1} \|\phi\|_{L_p}^p.
\end{aligned}
\end{equation*}
Thus, we have \eqref{eq: holder conti of T1-1}.

To obtain \eqref{eq: holder conti of T1-2}, use \eqref{ineq:slobodetskii's embedding 2} instead of \eqref{ineq:slobodetskii's embedding 1} and repeat the proof. \\

\noindent\textit{Proof of \eqref{item: holder conti of T2}}. 
Let $\rho>0$ and $\gamma>0$. Notice that \eqref{ineq:slobodetskii's embedding 1} yields
\begin{equation}
\label{start of proof of holder conti of T2 - 1}
\begin{aligned}
\left\| T_\rho^2 f - T_\gamma^2 f \right\|_{H^{2\theta/\alpha}_p}^p 
&\leq N |\rho-\gamma|^{\alpha\mu p - 1} \int_0^{T}\int_0^{T} 1_{t>s} \frac{\left\| T_{t}^2 f - T_{s}^2 f \right\|_{H^{2\theta/\alpha}_p}^p}{|t-s|^{1+\alpha\mu p}}dsdt.
\end{aligned}
\end{equation}
Then, by definition of $T_t^2f$ (see \eqref{def of t2}),
\begin{equation}
\label{start of proof of holder conti of T2 - 2}
\begin{aligned}
&\left\| T^2_{t} f - T^2_{s} f \right\|_{H_p^{2\theta/\alpha}} \\
&\quad\leq \left\| \int_{s}^{t}\int_{\bR^d} q_{\alpha,1}(t-r,y)f(r,\cdot-y)dydr\right\|_{H_p^{2\theta/\alpha}} \\
&\quad\quad + \left\|\int_0^{s}\int_{\bR^d}\left( q_{\alpha,1}(t-r,y) - q_{\alpha,1}(s-r,y) \right)f(r,\cdot-y)dydr \right\|_{H_p^{2\theta/\alpha}}.
\end{aligned}
\end{equation}
Set
\begin{equation}
\label{start of proof of holder conti of T2 - 3}
\begin{aligned}
I_1 &:= \int_0^{T}\int_0^{T} 1_{t>s}\frac{\left\| \int_{s}^{t}\int_{\bR^d} q_{\alpha,1}(t-r,y)f(r,\cdot-y)dydr\right\|_{H_p^{2\theta/\alpha}}^p}{|t-s|^{1+\alpha \mu p}}  dsdt,\\
I_2 &:= \int_0^{T}\int_0^{T} 1_{t>s}\frac{\left\|\int_0^{s}\int_{\bR^d}\left( q_{\alpha,1}(t-r,y) - q_{\alpha,1}(s-r,y) \right)f(r,\cdot-y)dydr \right\|_{H_p^{2\theta/\alpha}}^p}{|t-s|^{1+\alpha \mu p}} dsdt.
\end{aligned}
\end{equation}
Then, apply \eqref{start of proof of holder conti of T2 - 2} and \eqref{start of proof of holder conti of T2 - 3} to \eqref{start of proof of holder conti of T2 - 1},
\begin{equation}
\label{cal for holder regularity - deterministic part - 0}
\begin{aligned}
\left\| T_t^2 f - T_s^2 f \right\|_{H^{2\theta/\alpha}_p}^p &\leq  |t-s|^{\alpha\mu p - 1}(I_1 + I_2).
\end{aligned}
\end{equation}
To deal with $I_1$, we employ Minkowski's inequality, the change of variable, and Lemma \ref{lemma:properties of bessel potential spaces} \eqref{multiplier theorem}. Then,
\begin{equation}
\label{cal for holder regularity - deterministic part - 1}
\begin{aligned}
I_1 
&\leq \int_0^{T}t^{-1-\alpha \mu p}  \left( \int_{0}^{t}\int_{\bR^d}\left| \left( (1-\Delta)^{\theta/\alpha}q_{\alpha,1} \right)(r,y) \right| dy dr  \right)^p dt\int_0^{T}  \left\|f(s,\cdot)\right\|_{L_p}^p ds  \\
& \leq I_{11} + I_{12},
\end{aligned}
\end{equation}
where
\begin{equation*}
\begin{aligned}
I_{11} &:= \int_0^{T}t^{-1-\alpha \mu p}  \left( \int_{0}^{t}\int_{\bR^d}\left|q_{\alpha,1}(r,y)\right| dy dr  \right)^p dt\int_0^{T}  \left\|f(s,\cdot)\right\|_{L_p}^p ds, \\
I_{12} &:= \int_0^{T}t^{-1-\alpha \mu p}  \left( \int_{0}^{t}\int_{\bR^d} \left|\left( \Delta^{\theta/\alpha}q_{\alpha,1} \right)(r,y)\right| dy dr  \right)^p dt\int_0^{T}  \left\|f(s,\cdot)\right\|_{L_p}^p ds.
\end{aligned}
\end{equation*}
Because $\mu < 1$, Lemma \ref{lemma:properties of kernel} \eqref{item:scaling properties of kernel} and \eqref{item:upper bound properties of kernel_2} yield
\begin{equation}
\label{cal for holder regularity - deterministic part - 11}
\begin{aligned}
I_{11} 
&= \int_0^{T}t^{-1 - \alpha\mu p + \alpha p} dt \left( \int_{\bR^d}\left| q_{\alpha,1}(1,y) \right| dy  \right)^p \int_0^{T}  \left\|f(s,\cdot)\right\|_{L_p}^p ds \\
&\leq N   \int_0^{T}  \left\|f(s,\cdot)\right\|_{L_p}^p ds. \\
\end{aligned}
\end{equation}
Similarly, since $\mu < 1-\theta/\alpha$,
\begin{equation}
\label{cal for holder regularity - deterministic part - 12}
\begin{aligned}
I_{12} 
&= \int_0^{T}t^{-1 -\alpha \mu p  + \alpha p - \theta p} dt \left( \int_{\bR^d}\left| \left( \Delta^{\theta/\alpha}q_{\alpha,1} \right)(1,y) \right| dy  \right)^p \int_0^{T}  \left\|f(s,\cdot)\right\|_{L_p}^p ds \\
&\leq N  \int_0^{T}  \left\|f(s,\cdot)\right\|_{L_p}^p ds. \\
\end{aligned}
\end{equation}
Thus, by applying \eqref{cal for holder regularity - deterministic part - 11} and \eqref{cal for holder regularity - deterministic part - 12} to \eqref{cal for holder regularity - deterministic part - 1}, we have
\begin{equation}
\label{cal for holder regularity - deterministic part - end - 1}
I_{1} \leq N\int_0^T\|f(s,\cdot)\|_{L_p}^p ds.
\end{equation}
Next, we address $I_2$. Similar to the case for $I_1$, we have
\begin{equation}
\label{cal for holder regularity - deterministic part - 2}
\begin{aligned}
I_2 
&\leq \int_0^{T}\frac{\int_0^{T-t} \left(\int_0^{s}\int_{\bR^d}\left| (1-\Delta)^{\theta/\alpha}\left( q_{\alpha,1}(t+r,y) - q_{\alpha,1}(r,y) \right) \right|dy\|f(s - r,\cdot)\|_{L_p}dr \right)^p ds}{t^{1+\alpha \mu p}}dt \\
& \leq I_{21} + I_{22},
\end{aligned}
\end{equation}
where 
\begin{equation*}
\begin{aligned}
I_{21} &:= \int_0^{T}\frac{\int_0^{T-t} \left(\int_0^{s}\int_{\bR^d}\left| q_{\alpha,1}(t+r,y) - q_{\alpha,1}(r,y) \right|dy \|f(s - r,\cdot)\|_{L_p} dr \right)^p ds}{t^{1+\alpha \mu p}}dt, \\
I_{22} &:= \int_0^{T} \frac{\int_0^{T-t} \left(\int_0^{s}\int_{\bR^d}\left| \Delta^{\theta/\alpha}\left( q_{\alpha,1}(t+r,y) - q_{\alpha,1}(r,y) \right) \right|dy \|f(s - r,\cdot)\|_{L_p} dr \right)^p ds}{t^{1+\alpha \mu p}} dt.
\end{aligned}
\end{equation*}
Since $\mu < 1$, by Minkowski's inequality and the fundamental theorem of calculus, we have
\begin{equation}
\label{cal for holder regularity - deterministic part - 21}
\begin{aligned}
I_{21} 
&\leq \int_0^{T} t^{-1-\alpha \mu p}  \left(\int_0^{T-t}\int_{\bR^d}\left| q_{\alpha,1}(t+r,y) - q_{\alpha,1}(r,y) \right| dy dr \right)^p dt  \int_0^T\|f(s,\cdot)\|_{L_p(\bR^d)}^pds \\
&\leq \int_0^{T} t^{-1-\alpha \mu p}  \left(\int_0^{T-t}\int_r^{t+r}\int_{\bR^d}\left| q_{\alpha,2}(s,y) \right| dy ds dr \right)^p dt \int_0^T\|f(s,\cdot)\|_{L_p(\bR^d)}^pds \\
&\leq N \int_0^{T} t^{-1-\alpha \mu p}  \left(\int_0^{T-t}  r^{\alpha-1} - (t+r)^{\alpha-1}  dr \right)^p dt \int_0^T\|f(s,\cdot)\|_{L_p(\bR^d)}^pds \\
&\leq N \int_0^{T} t^{-1-\alpha \mu p + \alpha p} dt \int_0^T\|f(s,\cdot)\|_{L_p(\bR^d)}^pds \\
&\leq N  \int_0^T\|f(s,\cdot)\|_{L_p(\bR^d)}^pds. \\
\end{aligned}
\end{equation}
Additionally, since $\mu < 1-\theta/\alpha$, 
\begin{equation}
\label{cal for holder regularity - deterministic part - 22}
\begin{aligned}
I_{22} 
&\leq \int_0^{T} t^{-1-\alpha \mu p}  \left(\int_0^{T-t}\int_{\bR^d}\int_r^{t+r}\left| (\Delta^{\theta/\alpha}q_{\alpha,2}(s,y) \right| ds dy dr \right)^p dt \int_{0}^{T} \|f(s,\cdot)\|_{L_p}^p ds \\
&\leq N \int_0^{T} t^{-1-\alpha \mu p + \alpha p - \theta p } dt \int_{0}^{T} \|f(s,\cdot)\|_{L_p}^p ds \\
&\leq N \int_{0}^{T} \|f(s,\cdot)\|_{L_p}^p ds.
\end{aligned}
\end{equation}
Therefore, by employing \eqref{cal for holder regularity - deterministic part - 21} and \eqref{cal for holder regularity - deterministic part - 22} to \eqref{cal for holder regularity - deterministic part - 2}, we have
\begin{equation}
\label{cal for holder regularity - deterministic part - end - 2}
I_{2} \leq N\int_0^T\|f(s,\cdot)\|_{L_p}^p ds,
\end{equation}
and thus by combining \eqref{cal for holder regularity - deterministic part - end - 1} and \eqref{cal for holder regularity - deterministic part - end - 2} to \eqref{cal for holder regularity - deterministic part - 0}, we have \eqref{eq: holder conti of T2-1}. 

To obtain \eqref{eq: holder conti of T2-2}, employ \eqref{ineq:slobodetskii's embedding 2} instead of \eqref{ineq:slobodetskii's embedding 1} and repeat the proof word for word.

\vspace{2mm}

\noindent\textit{Proof of \eqref{item: holder conti of T3}}.  By \eqref{ineq:slobodetskii's embedding 1}, we have
\begin{equation*}
\begin{aligned}
\bE \left\| T_{\rho}^3 g - T_{\gamma}^3 g \right\|_{H^{2\theta/\alpha}_p(l_2)}^p  
&\leq N |\rho-\gamma|^{\alpha\mu p - 1} \int_0^{T}\int_0^{T} 1_{t>s} \frac{\bE\left\| T_{t}^3 g - T_{s}^3 g \right\|_{H^{2\theta/\alpha}_p(l_2)}^p}{|t-s|^{1+\alpha\mu p}}dsdt.
\end{aligned}
\end{equation*}
Notice that the Burkholder-Davis-Gundy and Minkowski's inequalities imply that
\begin{equation*} 
\begin{aligned}
&\bE\left\| T^3_{t} g - T^3_{s} g \right\|_{H_p^{2\theta/\alpha}(l_2)}^p \\
&\leq  N\int_{\bR^d} \bE \left( \sum_{k}\left| \int_{s}^{t}\int_{\bR^d} (\left( 1-\Delta \right)^{\theta/\alpha}q_{\alpha,\beta})(t-r,y)g^k(r,x-y) dydw_r^k  \right|^2 \right)^{p/2} \\
&\quad +  \left( \sum_{k}\left| \int_0^{s}\int_{\bR^d} \left[\left( 1-\Delta \right)^{\theta/\alpha}\left( q_{\alpha,\beta}(t-r,y)  - q_{\alpha,\beta}(s-r,y) \right)\right]g^k(r,x-y) dydw_r^k \right|^2 \right)^{p/2}dx \\
&\leq  N\int_{\bR^d} \bE \left( \int_0^{t-s}\left( \int_{\bR^d} \left|(\left( 1-\Delta \right)^{\theta/\alpha}q_{\alpha,\beta})(t-s-r,y)\right||g(s+r,x-y)|_{l_2}dy \right)^2dr \right)^{p/2} \\
&\quad +  \left(  \int_0^{s}\left( \int_{\bR^d}  \left|\left( 1-\Delta \right)^{\theta/\alpha}\left( q_{\alpha,\beta}(t-r,y)  - q_{\alpha,\beta}(s-r,y) \right)\right| |g(r,x-y)|_{l_2} dy \right)^2 dr \right)^{p/2}dx \\
&\leq N\bE\left( \int_0^{t-s} \left( \int_{\bR^d} \left|(\left( 1-\Delta \right)^{\theta/\alpha}q_{\alpha,\beta})(t-s-r,y)\right| dy \right)^2 \|g(s+r,\cdot)\|_{L_p(l_2)}^2   dr \right)^{p/2} \\
&\quad + N\bE\left( \int_0^{s} \left( \int_{\bR^d} \left|\left( 1-\Delta \right)^{\theta/\alpha}\left(q_{\alpha,\beta}(t-r,y) - q_{\alpha,\beta}(s-r,y)\right)\right| dy \right)^2 \| g(r,\cdot) \|_{L_p(l_2)}^2  dr \right)^{p/2}.
\end{aligned}
\end{equation*}
Then, set
\begin{equation*}
\begin{aligned}
I_3:=&\int_0^{T}\int_0^{T} 1_{t>s}|t-s|^{-1-\alpha\mu p}\bE\left( \int_0^{t-s} A(t,s,r) \| g(s+r,\cdot) \|_{L_p(l_2)}^2  dr \right)^{p/2} dsdt, \\
I_4:=&\int_0^{T}\int_0^{T} 1_{t>s}|t-s|^{-1-\alpha\mu p} \bE\left( \int_0^{s} B(t,s,r) \|g(r,\cdot)\|_{L_p(l_2)}^2  dr \right)^{p/2} dsdt,
\end{aligned}
\end{equation*}
where 
\begin{equation*}
\begin{aligned}
A(t,s,r) &= \left( \int_{\bR^d} |\left( 1-\Delta \right)^{\theta/\alpha} q_{\alpha,\beta}(t-s-r,y)| dy \right)^2, \\
B(t,s,r) &= \left( \int_{\bR^d} \left|\left( 1-\Delta \right)^{\theta/\alpha} \left( q_{\alpha,\beta}(t-r,y) - q_{\alpha,\beta}(s-r,y) \right)\right| dy \right)^2.
\end{aligned}
\end{equation*}
Note that Minkowski's inequality and Lemma \ref{lemma:properties of bessel potential spaces} \eqref{multiplier theorem} imply that
\begin{equation}
\label{cal for holder regularity - stochastic part - 1}
\begin{aligned}
I_3 
& \leq \int_0^{T} t^{-1-\alpha\mu p}\left( \int_0^{t} \left( \int_{\bR^d} \left| (\left( 1-\Delta \right)^{\theta/\alpha}q_{\alpha,\beta})(t-r,y) \right| dy \right)^2 dr \right)^{p/2}  dt \| g \|_{\bL_p(T,l_2)}^p\\
& \leq I_{31} + I_{32},
\end{aligned}
\end{equation}
where
\begin{equation*}
\begin{aligned}
I_{31} &:= \int_0^{T} t^{-1-\alpha\mu p}\left( \int_0^{t} \left( \int_{\bR^d} |q_{\alpha,\beta}(r,y)| dy \right)^2 dr \right)^{p/2}  dt \| g \|_{\bL_p(T,l_2)}^p, \\
I_{32} & := \int_0^{T} t^{-1-\alpha\mu p}\left( \int_0^{t} \left( \int_{\bR^d} \left| \left( \Delta^{\theta/\alpha}q_{\alpha,\beta}\right)(r,y) \right| dy \right)^2 dr \right)^{p/2}  dt \| g \|_{\bL_p(T,l_2)}^p.
\end{aligned}
\end{equation*}
Since $\frac{1}{\alpha}\left( \alpha - \beta + \frac{1}{2} \right) > \mu$, by Lemma \ref{lemma:properties of kernel} \eqref{item:upper bound properties of kernel_2}, we have
\begin{equation}
\label{cal for holder regularity - stochastic part - 31}
\begin{aligned}
I_{31} 
&= \int_0^{T} t^{-1-\alpha\mu p}\left( \int_0^{t} \left( \int_{\bR^d} |q_{\alpha,\beta}(r,y)| dy \right)^2 dr \right)^{p/2}  dt \| g \|_{\bL_p(T,l_2)}^p \\
&= \int_0^{T} t^{-1-\alpha\mu p}\left( \int_0^{t}  r^{2(\alpha-\beta)} dr \right)^{p/2}  dt \left( \int_{\bR^d} |q_{\alpha,\beta}(1,y)| dy \right)^p \| g \|_{\bL_p(T,l_2)}^p \\
&= \int_0^{T} t^{- 1 + ( \alpha - \beta + 1/2 - \alpha\mu )p}  dt \left( \int_{\bR^d} |q_{\alpha,\beta}(1,y)| dy \right)^p \| g \|_{\bL_p(T,l_2)}^p \\
&\leq N\|g\|_{\bL_p(T,l_2)}^p. \\
\end{aligned}
\end{equation}
Similarly, as $\frac{1}{\alpha}\left( \alpha - \beta - \theta + \frac{1}{2} \right)  > \mu$,
\begin{equation}
\label{cal for holder regularity - stochastic part - 32}
\begin{aligned}
I_{32} 
&\leq \int_0^{T} t^{- 1 + ( - \alpha\mu + \alpha - \beta - \theta + 1/2 )p}  dt \left( \int_{\bR^d} |( \Delta^{\theta/\alpha}q_{\alpha,\beta})(1,y)| dy \right)^p \| g \|_{\bL_p(T,l_2)}^p \\
&\leq N\| g \|_{\bL_p(T,l_2)}^p.
\end{aligned}
\end{equation}
Therefore, by employing \eqref{cal for holder regularity - stochastic part - 31} and \eqref{cal for holder regularity - stochastic part - 32} to \eqref{cal for holder regularity - stochastic part - 1}, we have
\begin{equation}
\label{cal for holder regularity - stochastic part - end - 1}
I_{3} \leq N\|g\|_{\bL_p(T,l_2)}^p.
\end{equation}

\vspace{2mm}

In the case of $I_{4}$, by Minkowski's inequality and Lemma \ref{lemma:properties of bessel potential spaces} \eqref{multiplier theorem}, we have
\begin{equation*}
\begin{aligned}
I_4
& \leq I_{41} + I_{42}.
\end{aligned}
\end{equation*}
Further,
\begin{equation*}
\begin{aligned}
I_{41} &:= \int_0^{T} t^{-1-\alpha\mu p}\bE\int_0^{T-t}  \left( \int_0^{s} \left( \int_{\bR^d} |C(t,r,y)| dy \right)^2 \|g(s-r,\cdot)\|_{L_p(l_2)}^2  dr \right)^{p/2} ds dt, \\
I_{42} & := \int_0^{T} t^{-1-\alpha\mu p}\bE\int_0^{T-t} \left( \int_0^{s} \left( \int_{\bR^d} |\Delta^{\theta/\alpha}C(t,r,y)| dy \right)^2\|g(s-r,\cdot)\|_{L_p(l_2)}^2 dr \right)^{p/2} ds dt,
\end{aligned}
\end{equation*}
where
$$ C(t,r,y) = q_{\alpha,\beta}(t+r,y) - q_{\alpha,\beta}(r,y).
$$
We address $I_{41}$. By Minkowski's inequality, the fundamental theorem of calculus, and Lemma \ref{lemma:properties of kernel}, we have
\begin{equation}
\label{cal for holder regularity - stochastic part - 41}
\begin{aligned}
I_{41} 
& \leq N\int_0^T t^{-1-\alpha\mu p}\left( \int_0^{T-t} \left( \int_{\bR^d} |q_{\alpha,\beta}(t+r,y) - q_{\alpha,\beta}(r,y)| dy \right)^2 dr \right)^{p/2}dt\| g \|^p_{\bL_p(T,l_2)} \\
& \leq N\int_0^T t^{-1-\alpha\mu p}\left( \int_0^{T-t} \left( \int_r^{t+r} \int_{\bR^d} | q_{\alpha,\beta+1}(s,y)| dy ds \right)^2 dr \right)^{p/2}dt\| g \|^p_{\bL_p(T,l_2)} \\
& \leq N H_1 \| g \|^p_{\bL_p(T,l_2)},
\end{aligned}
\end{equation}
where
\begin{equation}
\label{def of h1}
H_1 := \int_0^T t^{-1-\alpha\mu p}\left( \int_0^{T-t} \left( \int_r^{t+r}  s^{\alpha-\beta-1}ds \right)^2 dr \right)^{p/2}  dt
\end{equation}
and  $N = N(\alpha,\beta,d,p,T)$.
Next, we claim that 
\begin{equation}
\label{h1 is finite}
H_1<\infty
\end{equation}
To demonstrate \eqref{h1 is finite}, set
\begin{equation*}
\begin{aligned}
\chi(t) 
:= \int_0^{T-t} \left( \int_r^{t+r}  s^{\alpha-\beta-1}ds \right)^2 dr.
\end{aligned}
\end{equation*}
Furthermore,
\begin{equation}
\label{relation between H and chi}
H_1 = \int_0^T t^{-1-\alpha\mu p}(\chi(t))^{p/2}dt
\end{equation}
Depending on the range of $\alpha-\beta$, we consider the following five cases.

\vspace{2mm}

\noindent\textit{(Case 1.)} $-1/2 < \alpha - \beta < 0$ 

For $t\in(0,T)$, we have
\begin{equation}
\label{bound of h - 1}
\begin{aligned}
0\leq \chi(t)
\leq N\int_0^{T-t} r^{2(\alpha-\beta)} - (t+r)^{2(\alpha-\beta)} dr 
\leq Nt^{2(\alpha-\beta)+1},
\end{aligned}
\end{equation}
where $N = N(\alpha,\beta)$.
Then, since $\frac{1}{\alpha}(\alpha-\beta+1/2) > \mu$, by combining \eqref{relation between H and chi} and \eqref{bound of h - 1},
\begin{equation*}
\begin{aligned}
H_1
\leq N(\alpha,\beta,p)\int_0^{T} t^{-1+p(-\alpha\mu  + \alpha-\beta + 1/2)} dt < \infty.
\end{aligned}
\end{equation*}

\vspace{2mm}

\noindent\textit{(Case 2.)} $\alpha - \beta = 0$

Notice that
\begin{equation*}
\begin{aligned}
\chi(t) = \int_0^{T-t} \left( \int_r^{t+r}  s^{-1}ds \right)^2 dr = \int_0^{T-t} \left( \log\left( \frac{t+r}{r} \right) \right)^2 dr.
\end{aligned}
\end{equation*}
Obviously, $\chi(0) = 0$. Note that
\begin{equation*}
\begin{aligned}
\chi'(t) 
\leq  2\int_0^{\infty}  \frac{1}{1+r} \log\left( \frac{1+r}{r} \right) dr 
=  2\int_0^\infty x(e^{x}-1)^{-1}dx 
= \pi^2/3 
\end{aligned}
\end{equation*}
on $t\in(0,T/2)$. Thus,
\begin{equation}
\label{bound of h - 2}
\begin{aligned}
0\leq \chi(t) = \chi(t) - \chi(0) = \int_0^t \chi'(s)ds \leq \frac{\pi^2}{3}t
\end{aligned}
\end{equation}
on $t\in(0,T/2)$.
Additionally, $\chi(t) \leq N$ on $t\in(T/2,T)$. Therefore, 
\begin{equation*}
H_1 
\leq N\int_0^{T/2} t^{-1+(-\alpha\mu+1/2) p}  dt + N <\infty,
\end{equation*}
where $N = N(\alpha,\beta,p)$.

\vspace{2mm}

\noindent\textit{(Case 3.)} $0 < \alpha - \beta < 1/2$

Observe that $\chi$ is twice continuously differentiable, and
\begin{equation*}
\begin{aligned}
\chi'(t) 
&= (\alpha-\beta)^{-2}\left[ -T^{2(\alpha-\beta)} - (T-t)^{2(\alpha-\beta)} + 2T^{\alpha-\beta}(T-t)^{\alpha-\beta} \right] \\
&\quad + 2(\alpha-\beta)^{-1}\int_0^{T-t}(t+r)^{2(\alpha-\beta)-1} - (t+r)^{\alpha-\beta-1}r^{\alpha-\beta} dr,
\end{aligned}
\end{equation*}
and
\begin{equation}
\label{second derivative of h}
\begin{aligned}
\chi''(t) 
&= 2(\alpha-\beta)^{-1}(T-t)^{2(\alpha-\beta)-1} - 2(\alpha-\beta)^{-1}T^{\alpha-\beta}(T-t)^{\alpha-\beta-1} \\
&\quad + 2(\alpha-\beta)^{-1} T^{\alpha-\beta-1}(T-t)^{\alpha-\beta} - 2(\alpha-\beta)^{-1}t^{2(\alpha-\beta)-1} \\
&\quad  - 2(\alpha-\beta)^{-1}(\alpha-\beta-1)\int_0^{T-t}(t+r)^{\alpha-\beta-2}r^{\alpha-\beta}dr.
\end{aligned}
\end{equation}
In addition, $\chi(0) = \chi'(0) = 0$. Then, by using the fundamental theorem of calculus and $\alpha-\beta\in(0,1/2)$, we obtain
\begin{equation}
\label{estimate of chi}
\begin{aligned}
\chi(t) 
& = \int_0^t\int_0^s \chi''(\rho)d\rho ds \\
& \leq 2(\alpha-\beta)^{-1}\int_0^t\int_0^s 
(T-\rho)^{2(\alpha-\beta)-1}  d\rho ds\\
&\quad - 2(\alpha-\beta)^{-1}(\alpha-\beta-1)  \int_0^t\int_0^s \int_0^{T-\rho}(\rho+r)^{\alpha-\beta-2}r^{\alpha-\beta}dr
d\rho ds \\
& \leq N \int_0^t\left( T^{2(\alpha-\beta)} - (T-s)^{2(\alpha-\beta)} \right) ds + N \int_0^t\int_0^s \int_0^{T-\rho}(\rho+r)^{2(\alpha-\beta)-2}drd\rho ds \\
& \leq N \int_0^t  s^{2(\alpha-\beta)}  ds + N \int_0^t\int_0^s T^{2(\alpha-\beta)-1}drd\rho ds \\
& \leq Nt^{2(\alpha-\beta)+1},
\end{aligned}
\end{equation}
where $N = N(\alpha,\beta,T)$. Thus, 
\begin{equation*}
H_1 
\leq N(\alpha,\beta,p,T)\int_0^{T} t^{-1+p(-\alpha\mu + \alpha - \beta +1/2)}  dt <\infty.
\end{equation*}

\vspace{2mm}

\noindent\textit{(Case 4.)} $\alpha - \beta = 1/2$

Because $\chi(0) = \chi'(0) = 0$, by the fundamental theorem of calculus, we have
\begin{equation*}
\begin{aligned}
\chi(t) 
& = \int_0^t\int_0^s \chi''(\rho)d\rho ds \\
& \leq N\int_0^t\int_0^s 
(T-\rho)^{1/2} + N\int_0^{T-\rho}(\rho+r)^{-3/2}r^{1/2}dr
d\rho ds \\
&\leq Nt^{2}(1+|\log t|),
\end{aligned}
\end{equation*}
where $N = N(T)$. Therefore, 
\begin{equation*}
H_1 
\leq N(p,T)\int_0^{T} t^{-1+p(-\alpha\mu + 1)}(1+|\log t|)^{p/2}  dt <\infty.
\end{equation*}

\vspace{2mm}
\newpage

\noindent\textit{(Case 5.)} $\alpha - \beta  > 1/2$

Similar to before, $\chi(0) = \chi'(0) = 0$. Additionally, as in \eqref{estimate of chi}, we have
\begin{equation*}
\begin{aligned}
\chi(t) 
& = \int_0^t\int_0^s \chi''(r)drds \\
&\leq N \int_0^t\left( T^{2(\alpha-\beta)} - (T-s)^{2(\alpha-\beta)} \right) ds + N \int_0^t\int_0^s \int_0^{T-\rho}(\rho+r)^{2(\alpha-\beta)-2}drd\rho ds \\
& \leq N \int_0^t  s  ds + N \int_0^t\int_0^s  d\rho ds \\
&\leq Nt^2,
\end{aligned}
\end{equation*}
where $N = N(\alpha,\beta,T)$. Therefore,
\begin{equation*}
H_1 
\leq N(\alpha,\beta,p,T)\int_0^{T} t^{-1+p(-\alpha\mu + 1)}  dt <\infty.
\end{equation*}
Thus, we have \eqref{h1 is finite}. Then, by combining \eqref{h1 is finite} and \eqref{cal for holder regularity - stochastic part - 41}, we have $I_{41} \leq N\| g \|^p_{\bL_p(T,l_2)}$.

Next, we deal with $I_{42}$. Minkowski's inequality, the fundamental theorem of calculus, and Lemma \ref{lemma:properties of kernel} yield
\begin{equation*}
\begin{aligned}
I_{42} 
& \leq N\int_0^T t^{-1-\alpha\mu p}\left( \int_0^{T-t} \left( \int_{\bR^d} |\Delta^{\theta/\alpha}q_{\alpha,\beta}(t+r,y) - \Delta^{\theta/\alpha}q_{\alpha,\beta}(r,y)| dy \right)^2 dr \right)^{p/2}dt\| g \|^p_{\bL_p(T,l_2)} \\
& \leq N\int_0^T t^{-1-\alpha\mu p}\left( \int_0^{T-t} \left( \int_r^{t+r} \int_{\bR^d} |\partial_s \Delta^{\theta/\alpha} q_{\alpha,\beta}(s,y)| dy ds \right)^2 dr \right)^{p/2}dt\| g \|^p_{\bL_p(T,l_2)} \\
& \leq N H_2 \| g \|^p_{\bL_p(T,l_2)},
\end{aligned}
\end{equation*}
where $N = N(\alpha,\beta,d,p,T)$ and $H_2 := \int_0^T t^{-1-\alpha\mu p}\left( \int_0^{T-t} \left( \int_r^{t+r}  s^{\alpha-\beta-\theta-1}ds \right)^2 dr \right)^{p/2}  dt$. Similar to the case of $H_1$ (see \eqref{def of h1} and \eqref{h1 is finite}), we demonstrate that
\begin{equation*}
H_2<\infty
\end{equation*}
by considering five cases for $\alpha - \beta - \theta$ instead of $\alpha - \beta$. Then, we have $I_{42}\leq N\| g \|^p_{\bL_p(T,l_2)}$. The lemma is proved.

\end{proof}

\begin{proof}[\textbf{Proof of Theorem \ref{thm:embedding theorem for sol space}}]
It suffices to show that the assertion holds for $\tau = T$. Indeed, assume the results holds for $\cH_p^\gamma(T)$, and let $\tau\leq T$ be a bounded stopping time and $u\in\cH_{p}^{\gamma}(\tau)$. Then, by Definition \ref{eq:def_of_sol} for $\ep>0$, there exists $(f,g)\in \bH_p^{\gamma-2}(\tau)\times\bH_p^{\gamma-2+c_0}(\tau,l_2)$ such that
\begin{equation*}
\partial_t^\alpha u =  f + \partial_t^\beta\int_0^t  g^k dw_t^k\,;\quad u(0,\cdot) = u_0(\cdot)
\end{equation*}
and
\begin{equation*}
\| u \|_{\bH_p^{\gamma}(\tau)} + \| u_0 \|_{U_p^{\gamma}} + \| f \|_{\bH_p^{\gamma-2}(\tau)} + \| g \|_{\bH_p^{\gamma-2+c_0}(\tau,l_2)} \leq \| u \|_{\cH_p^\gamma(\tau)} + \ep.
\end{equation*}
Set $\bar f := (f - \Delta u)1_{t\leq \tau}$ and $\bar g := g1_{t\leq \tau}$; thus, $u$ satisfies 
\begin{equation} 
\label{eq:conti of u to extend T}
\partial_t^\alpha u = \Delta u + \bar f + \partial_t^\beta\int_0^t \bar g^k dw_t^k,\quad 0<t\leq \tau\,;\quad u(0,\cdot) = u_0(\cdot).
\end{equation}
In contrast, by \cite[Theorem 2.18]{kim2020sobolev}, there exists $v\in \cH_p^{\gamma}(T)$ such that $v$ satisfies 
\begin{equation*}
\partial_t^\alpha v = \Delta v + \bar f + \partial_t^\beta\int_0^t \bar g^k dw_t^k, \quad 0 < t \leq \tau\,;\quad v(0,\cdot) = u_0(\cdot).
\end{equation*}
Additionally, 
\begin{equation} 
\label{embedding estimate}
\begin{aligned}
\| v \|_{\cH_p^{\gamma}(T)} 
&\leq N\left( \left\| \bar f \right\|_{\bH_p^{\gamma-2}(T)} + \left\| \bar g \right\|_{\bH_p^{\gamma-2+c_0}(T,l_2)} + \left\| u_0 \right\|_{U_p^{\gamma}} \right) \\
&\leq N\left( \| u \|_{\bH_p^{\gamma}(\tau)} + \| u_0 \|_{U_p^{\gamma}} + \| f \|_{\bH_p^{\gamma-2}(\tau)} + \| g \|_{\bH_p^{\gamma-2+c_0}(\tau,l_2)} \right)  \\
&\leq  N\| u \|_{\cH_p^\gamma(\tau)} + N\ep,
\end{aligned}
\end{equation}
where $N$ is independent of $\ep$. Because $v\in \cH_p^{\gamma}(T)$, $v\in C^{\alpha\mu - 1/p}\left( [0,T];H_p^{\gamma-2\nu} \right)$ almost surely and 
\begin{equation}
\label{embedding for v}
\bE\| v \|_{C^{\alpha\mu - 1/p}([0,T];H_p^{\gamma-2\nu})}^p \leq N\| v \|_{\cH_p^{\gamma}(T)}^p
\end{equation}
by the hypothesis. Therefore, due to $\tau\leq T$, \eqref{embedding estimate} and \eqref{embedding for v} yield
\begin{equation}
\label{ineq: proof of continuous version}
\bE\| u \|_{C^{\alpha\mu - 1/p}([0,\tau];H_p^{\gamma-2\nu})}^p 
\leq\bE\| v \|_{C^{\alpha\mu - 1/p}([0,T];H_p^{\gamma-2\nu})}^p
\leq N\| v \|_{\cH_p^{\gamma}(T)}^p 
\leq N\| u \|_{\cH_p^{\gamma}(\tau)}^p + N\ep,
\end{equation}
where $N$ is independent of $\ep$.
Note that $\bar u:=u-v$ satisfies
\begin{equation*}
\partial_t^\alpha \bar u = \Delta \bar u,\quad 0< t \leq \tau;\quad \bar u(0,\cdot) = 0.
\end{equation*}
Then, by the deterministic version of \cite[Theorem 2.18]{kim2020sobolev}, we have $u(t,\cdot)=v(t,\cdot)$ for almost every $(\omega,t)\in\Omega\times[0,\tau]$.  Thus, $v$ is an $H_p^{\gamma-2\nu}$-valued continuous version of $u$. Additionally, from \eqref{ineq: proof of continuous version}
\begin{equation*}
\bE\| u \|_{C^{\alpha\mu - 1/p}([0,\tau];H_p^{\gamma-2\nu})}^p
\leq N\| u \|_{\cH_p^{\gamma}(\tau)}^p + N\ep.
\end{equation*}
In addition, as $\ep>0$ is arbitrary and $N$ is independent of $\ep$, we have \eqref{item:time holder continuity of sol}. Additionally, notice that \eqref{embedding estimate} allows us to prove the assertions with $\left\|  f \right\|_{\bH_p^{\gamma-2}(T)} + \left\|  g \right\|_{\bH_p^{\gamma-2+c_0}(T,l_2)} + \left\| u_0 \right\|_{U_p^{\gamma}}$ instead of $\| u \|_{\cH_p^{\gamma}(T)}$. 

Due to Lemma \ref{lemma:properties of bessel potential spaces} \eqref{item:isometry}, we only consider the case $\gamma = 2\nu$, where $\nu\in(0,1)$ satisfies \eqref{conditions of mu and nu}. Moreover, by using the approximation to the identity, we may assume that $u_0$ is infinitely differentiable and compactly supported in $x$. Furthermore, we also assume that $f$ and $g = (g^1,g^2,\dots)$ denotes the function of the form satisfying \eqref{form of g} (see \cite[Theorem 3.10]{kry1999analytic}). 
Additionally, it should be remarked that $u$ can be written as
\begin{equation*}
u(t,x) = T^1_t u_0 + T^2_{t}f + T^3_{t}g
\end{equation*}
since $u$ satisfies
\begin{equation*}
\partial_t^\alpha u = \Delta u + f + \partial_t^\beta\int_0^t g^k dw_t^k\,;\quad u(0,\cdot) = u_0(\cdot)
\end{equation*}
for almost every $(\omega,t)\in \opar0,\tau\cbrk$ (see \cite{kim2017lq,kim2019sobolev,kim2015asymptotic}).

Now, observe that $u\in C^{\alpha\mu - 1/p}([0,\tau];H_p^{\gamma-2\nu})$ almost surely. Indeed, by Lemma \ref{holder semi norm control lemma}, we have
\begin{equation*}
\begin{aligned}
\bE\left\| T^1_{t} f - T^1_{s} f \right\|_{H^{\gamma-2\nu}_p}^p
&\leq N|t-s|^{\alpha\mu p - 1}\left\| \phi \right\|_{H^\gamma_p}^p,\\
\bE\left\| T^2_{t} f - T^2_{s} f \right\|_{H^{\gamma-2\nu}_p}^p
&\leq N|t-s|^{\alpha\mu p - 1}\left\| f \right\|_{\bH^{\gamma-2}_p(\tau)}^pdr,\\
\bE \left\| T_{t\wedge \tau}^3 g - T_{s\wedge \tau}^3 g \right\|_{H^{\gamma-2\nu}_p(l_2)}^p 
&\leq N|t-s|^{\alpha \mu p - 1}\| g \|_{\bH^{\gamma-2+c_0}_p(\tau,l_2)}^p.
\end{aligned}
\end{equation*}
Thus, Kolmogorov continuity theorem (e.g., \cite[Theorem 1.4.8]{krylov2002introduction}) and the above inequalities again, we have the continuity of $u$.

Next we prove \eqref{item:time holder continuity of sol}. Choose $\mu$ and $\nu$ satisfy \eqref{conditions of mu and nu}. Observe that Lemma \ref{lemma:properties of bessel potential spaces} \eqref{item:isometry} and \eqref{eq: holder conti of T1-1} with $\theta = \alpha - \alpha\nu$ imply that
\begin{equation*}
\begin{aligned}
\bE\sup_{0\leq s < t \leq T}\frac{\| T_t^1 u_0 - T_s^1 u_0 \|_{H^{\gamma - 2\nu}_p}^p}{|t-s|^{\alpha\mu p - 1}} 
\leq  N\bE\left\| u_0 \right\|_{H^{\gamma }_p}^p.
\end{aligned}
\end{equation*}
As \eqref{conditions of mu and nu} is assumed, Lemma \ref{lemma:properties of bessel potential spaces} \eqref{item:isometry} and \eqref{eq: holder conti of T2-2}  with $\theta = \alpha -\alpha \nu$ yield the following:
\begin{equation*}
\begin{aligned}
\bE\sup_{0\leq s < t \leq \tau} \frac{\| T^2_t f - T^2_2 f \|_{H^{\gamma- 2\nu}_p}^p}{|t-s|^{\alpha\mu p - 1}} 
&\leq N \| f \|_{\bH_p^{\gamma  - 2}(\tau)}^p.
\end{aligned}
\end{equation*}
Furthermore, by \eqref{conditions of mu and nu}, Lemma \ref{lemma:properties of bessel potential spaces} \eqref{item:isometry}, and \eqref{eq: holder conti of T3-2} with $\theta = \alpha(1-c_0/2)-\alpha\nu$, we have
\begin{equation*}
\begin{aligned}
\bE \sup_{0\leq s < t \leq T}\frac{\left\| T_t^3 g - T_s^3 g \right\|_{H^{\gamma - 2\nu}_p(l_2)}^p}{|t-s|^{\alpha \mu p - 1}} 
&\leq N \| g \|_{\bH_p^{\gamma - 2 + c_0}(\tau,l_2)}^p.
\end{aligned}
\end{equation*}
The theorem is proved.

\end{proof}


\bibliographystyle{plain}

\end{document}